\documentclass[a4paper]{article}

\makeatletter

\@addtoreset{equation}{section}
\makeatother

\usepackage{cite}
\usepackage{comment}

\usepackage{amsfonts, amsmath, amsthm, amssymb, graphicx}

\graphicspath{{fig/}}

\newtheorem{lemma}{Lemma}
\newtheorem{theorem}{Theorem}
\newtheorem{conjecture}{Conjecture}
\newtheorem{cor}{Corollary}
\newtheorem{proposition}{Proposition}
\newtheorem{remark}{Remark}

\newcommand{\restr}[2]{\left. #1 \right|_{#2}}
\newcommand{\sgn}{\operatorname{sgn}\nolimits}
\newcommand{\hypoth}[2]{\medskip\noindent{#1}{\em #2}\medskip}
\newcommand{\ddef}[1]{{#1}}

\newcommand{\am}{\operatorname{am}\nolimits}
\newcommand{\der}[2]{\frac{d \, #1}{d \, #2} }
\newcommand{\pder}[2]{\frac{\partial \, #1}{\partial \, #2} }
\newcommand{\sn}{\operatorname{sn}\nolimits}
\newcommand{\cn}{\operatorname{cn}\nolimits}
\newcommand{\dn}{\operatorname{dn}\nolimits}
\newcommand{\E}{\operatorname{E}\nolimits}

\newcommand{\eq}[1]{$(\protect\ref{#1})$}
\newcommand{\be}[1]{\begin{equation}\label{#1}}
\newcommand{\ee}{\end{equation}}
\newcommand{\vect}[1]{\left( \begin{array}{c} #1 \end{array} \right)}
\newcommand{\Id}{\operatorname{Id}\nolimits}
\newcommand{\MAX}{\operatorname{MAX}\nolimits}
\newcommand{\SH}{\operatorname{SH}\nolimits}
\newcommand{\SO}{\operatorname{SO}\nolimits}
\newcommand{\SL}{\operatorname{SL}\nolimits}
\newcommand{\SE}{\operatorname{SE}\nolimits}
\newcommand{\Exp}{\operatorname{Exp}\nolimits}

\newcommand{\Diff}{\operatorname{Diff}\nolimits}
\renewcommand{\Vec}{\operatorname{Vec}\nolimits}
\newcommand{\Sym}{\operatorname{Sym}\nolimits}
\newcommand{\spann}{\operatorname{span}\nolimits}
\newcommand{\ad}{\operatorname{ad}\nolimits}
\newcommand{\map}[3]{#1 \, : \, #2 \to #3}  
\newcommand{\mapto}[3]{#1 \, : \, #2 \mapsto #3}  
\newcommand{\const}{\operatorname{const}\nolimits}
\newcommand{\cl}{\operatorname{cl}\nolimits}

\newcommand{\pdder}[2]{\frac{\partial^2 #1}{\partial \, {#2}^2} }
\def\Ho{$(\mathbf{H1})$} 
\def\Ht{$(\mathbf{H2})$} 
\def\Hth{$(\mathbf{H3})$} 
\def\Hf{$(\mathbf{H4})$} 
\def\tconj{t^1_{\operatorname{conj}}}

\def\tmax{t_{\MAX}^1}

\def\a{\alpha}
\def\tcut{t_{\operatorname{cut}}}
\def\tmax{t_{\MAX}^1}

\def\then{\quad\Rightarrow\quad}
\def\D{\Delta}
\def\d{\delta}
\def\lam{\lambda}
\def\R{{\mathbb R}}
\def\ds{\displaystyle}
\def\eps{\varepsilon}

\def\ta{\widetilde{a}}
\def\tx{\widetilde{x}}
\def\tf{\widetilde{f}}
\def\tg{\widetilde{g}}
\def\tJ{\widetilde{J}}
\def\hg{\widehat{g}}
\def\htt{\widehat{t}}
\def\hlam{\widehat{\lambda}}
\def\hmu{\widehat{\mu}}
\def\bu{\bar{u}}
\def\bk{\bar{k}}
\def\bt{\bar{t}}
\def\blam{\bar{\lambda}}
\def\bth{\bar{\theta}}
\def\bh{\bar{h}}
\def\f{\varphi}
\def\vh{\vec h}
\def\vH{\vec H}

\def\Z{{\mathbb Z}}

\def\a{\alpha}
\def\b{\beta}
\def\p{\psi}
\def\t{\theta}
\def\dh{\dot h}

\def\g{\mathfrak{g}}

\def\ss{\sn p}
\def\cc{\cn p}
\def\dd{\dn p}

\newcommand{\twofiglabelsize}[8]
{
\begin{figure}[htbp]
\hfill
\includegraphics[height=#4cm]{#1}
\hfill
\includegraphics[height=#8cm]{#5}
\hfill
\\
\hfill
\parbox[t]{0.45\textwidth}{\caption{#2}\label{#3}}
\hfill
\parbox[t]{0.45\textwidth}{\caption{#6}\label{#7}}
\hfill
\end{figure}
}

\newcommand{\onefiglabelsize}[4]
{
\begin{figure}[htbp]
\begin{center}
\includegraphics[width=#4\textwidth]{#1}
\\
\parbox[t]{#4\textwidth}{\caption{#2}\label{#3}}
\end{center}
\end{figure}
}

\title
{Conjugate time in sub-Riemannian problem \\on Cartan group\footnote{This work is supported by the Russian Science Foundation 
under grant 17-11-01387-P and performed in Ailamazyan Program Systems Institute 
of Russian Academy of Sciences} 
}

\author{Yu. L. Sachkov\\
Program Systems Institute of RAS\\
Pereslavl-Zalessky,	 Russia\\
e-mail: yusachkov@gmail.com}

\begin{document}

\maketitle

\begin{abstract}
The Cartan group is the free nilpotent Lie group of rank 2 and step~3. We consider the left-invariant sub-Riemannian problem on the Cartan group defined by an inner product in the first layer of its Lie algebra.
This problem gives a nilpotent approximation of an arbitrary sub-Riemannian problem with the growth vector $(2,3,5)$. 

In previous works we described a group of symmetries of the sub-Riemannian problem on the Cartan group, and the corresponding Maxwell time --- the first time when symmetric geodesics intersect one another. It is known that geodesics are not globally optimal after the Maxwell time. 

In this work we study local optimality of geodesics on the Cartan group. We prove that the first conjugate time along a geodesic is not less than the Maxwell time corresponding to the group of symmetries. We characterize geodesics for which the first conjugate time is equal to the first Maxwell time. Moreover, we describe continuity of the first conjugate time near infinite values.
\end{abstract}

\newpage
\tableofcontents

\newpage

\section{Introduction}
\label{sec:intro}

\subsection{Problem statement}
This work deals with the nilpotent sub-Riemannian problem  with the growth vector $(2, 3, 5)$.
This problem evolves on the {\em Cartan group}, which is
the connected simply connected free nilpotent Lie group of rank 2 and step 3.

The Lie algebra of the Cartan group is the 5-dimensional nilpotent Lie algebra $\g = \spann(X_1, X_2, X_3, X_4, X_5)$ with the multiplication table
$$
[X_1, X_2] = X_3, \quad [X_1, X_3] = X_4, \quad  [X_2, X_3] = X_5, \quad \ad X_4 = \ad X_5 = 0,
$$
see Fig. \ref{fig:cartan}.

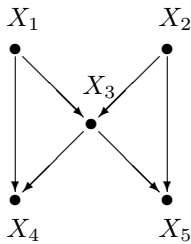
\begin{figure}[htb]
\setlength{\unitlength}{1cm}

\begin{center}
\begin{picture}(4, 4)
\put(1.1, 2.9){ \vector(1, -1){0.8}}
\put(1, 2.9){ \vector(0, -1){1.8}}
\put(2.9, 2.9){ \vector(-1, -1){0.8}}
\put(3, 2.9){ \vector(0, -1){1.8}}
\put(1.9, 1.9){ \vector(-1, -1){0.8}}
\put(2.1, 1.9){ \vector(1, -1){0.8}}

\put(1, 1){ \circle*{0.15}}
\put(1, 3){ \circle*{0.15}}
\put(3, 1){ \circle*{0.15}}
\put(3, 3){ \circle*{0.15}}
\put(2, 2){ \circle*{0.15}}

\put(1, 0.5) {$X_4$}
\put(3, 0.5) {$X_5$}
\put(1, 3.3) {$X_1$}
\put(3, 3.3) {$X_2$}
\put(2, 2.4) {$X_3$}

\end{picture}

\caption{ The Cartan Lie algebra $\g$\label{fig:cartan}}

\end{center}
\end{figure}

We consider the sub-Riemannian problem on the Cartan group $G$ for the left-invariant sub-Riemannian structure generated by the orthonormal frame $X_1$, $X_2$:
\begin{align}
&\dot g = u_1 X_1(g) + u_2 X_2(g), \qquad g \in G, \quad u = (u_1, u_2) \in \R^2,  \label{SR1}\\
& g(0) = g_0, \quad g(t_1) = g_1, \label{SR2}\\
& l = \int_0^{t_1} \sqrt{u_1^2 + u_2^2} \, d t \rightarrow \min, \label{SR}
\end{align}
see books~\cite{mont, ABB} as a reference on sub-Riemannian geometry.
 
In appropriate coordinates $g = (x,y,z,v,w)$ on the Cartan group $G \cong \R^5$, the problem 
 is stated as follows:
\begin{align} 
&\dot{g} = \vect{\dot{x} \\ \dot{y} \\ \dot{z} \\ \dot{v} \\ \dot{w}} = u_1 
\vect{1 \\ 0 \\ - \frac{y}{2} \\ 0\\ - \frac{x^2 + y^2}{2}} 
+ u_2 \vect{0 \\ 1 \\ \frac{x}{2} 
\\ \frac {x^2+y^2}{2}\\0}, \quad g \in \R^5, \quad u \in \R^2, \label{pr1}\\
& g(0) = g_0, \quad g(t_1) = g_1, \label{pr2} \\
& l = \int_0^{t_1} \sqrt{u_1^2 + u_2^2} \, d t \rightarrow \min.
\label{pr3}
\end{align}
Admissible trajectories $g(t)$ are Lipschitzian, and admissible controls $u(t)$ are measurable and locally bounded.

Since the problem is invariant under left shifts on the Cartan group, we can assume that the initial point is identity of the group: $g_0  = \Id =  (0, 0, 0, 0, 0)$.

Problem \eq{pr1}--\eq{pr3} has the following geometric model called a generalized Dido's problem.
Take two points $(x_0, y_0), \, (x_1, y_1) \in \R^2$ connected by a curve $\gamma_0 \subset \R^2$, a number $S\in \R$, and a point $c \in \R^2$. One has to find the shortest curve $\gamma \subset \R^2$ that connects the points $(x_0, y_0)$ and $(x_1, y_1)$, such that the domain bounded by $\gamma_0$ and $\gamma$ has algebraic area $S$ and center of mass $c$. 

Optimal control problem \eq{pr1}--\eq{pr3} is the sub-Riemannian (SR)  length minimization problem for the distribution 
\begin{align*}
&\D_g = \spann(X_1(g), X_2(g)), \qquad g \in G, \\
&X_1 = \pder{}{x}  - \frac{y}{2} \pder{}{z}  - \frac 
{x^2+y^2}{2} \pder{}{w},
\qquad
 X_2 = \pder{}{y} + \frac{x}{2} \pder{}{z} + \frac 
{x^2+y^2}{2} \pder{}{v},
\end{align*}
endowed with the inner product $\langle \, \cdot \, , \, \cdot \ \rangle$ in which $X_1$, $X_2$ are orthonormal:
$$
\langle X_i, X_j\rangle = \d_{ij}, \qquad i, j = 1, 2.
$$
The distribution $\D$ has the flag
$$
\D \subset \D^2 = \D + [\D, \D] \subset \D^3 = \D^2 + [\D, \D^2] = T G
$$
and the growth vector $(2,3,5) = (\dim \D_g, \dim \D^2_g, \dim \D^3_g)$, where
\begin{align*}
&\D^2_g = \spann(X_1(g), X_2(g),X_3(g)), \qquad \D^3_g = \spann(X_1(g), \dots,X_5(g)),\\
&X_3 = [X_1,X_2]  = \pder{}{z}
+
x \pder{}{v} + y \pder{}{w}, \\
&X_4 = [X_1, X_3]  =\pder{}{v}, \qquad
X_5 = [X_2, X_3]  =\pder{}{w}.
\end{align*}
Thus $(\D, \langle \, \cdot \, , \, \cdot \ \rangle)$ is a nilpotent SR structure with the growth vector $(2,3,5)$. It is a local quasihomogeneous nilpotent approximation~\cite{agr_sar, bellaiche, ABB} to an arbitrary SR structure on a 5-dimensional manifold with the growth vector $(2,3,5)$. Examples of such structures include models for:
\begin{itemize}
\item 
a pair of bodies rolling one on another without slipping or twisting \cite{cdc99, marigo_bicchi}, in particular, the sphere rolling on a plane \cite{jurd_plate-ball},
\item
a car with 2 off-hooked trailers \cite{laumond, vendit_laum_oriolo},
\item
electric charge moving in a magnetic field \cite{meneses_perez}.
\end{itemize}
Such a nilpotent SR structure is unique, up to homomorphism of the Lie group~$G$. Generalized Dido's problem~\eq{pr1}--\eq{pr3} is a model of the nilpotent SR problem with the growth vector $(2,3,5)$, see other models in \cite{sym, meneses_perez}. 

\medskip

The paper continues the study of this problem started in works~\cite{sym, dido_exp, max1, max2, max3}. The main result of these works is an upper bound of the cut time (i.e., the time of loss of { \em global} optimality) along extremal trajectories of the problem. The aim of this paper is to investigate the first conjugate time (i.e., the time of loss of {\em local} optimality) along the trajectories. We show that the function that gives the upper bound of the cut time provides the lower bound of the first conjugate time. In order to state this main result exactly, we recall necessary facts from the previous works~\cite{sym, dido_exp, max1, max2, max3}.

\subsection{Previously obtained results}\label{subsec:previous}

Problem \eq{pr1}--\eq{pr3} was considered first by R.~Brockett and L.~Dai~\cite{brock_dai}: they proved integrability of geodesics in this problem in Jacobi's functions.

The following results were obtained in~\cite{sym, dido_exp, max1, max2, max3}, if not stated otherwise.

Existence of optimal solutions of problem~\eq{pr1}--\eq{pr3} is implied by the Ra\-shev\-sky-Chow and Filippov theorems~\cite{notes}. 

\subsubsection{Pontryagin maximum principle}
By the Cauchy-Schwartz inequality, the sub-Riemannian length minimization problem~\eq{pr3} is equivalent to the energy minimization problem:
\be{J}
\int_0^{t_1} \frac{u_1^2+u_2^2}{2} \, d t \rightarrow \min.
\ee 
The Pontryagin maximum principle~\cite{PGBM, notes} was applied to the resulting optimal control problem~\eq{pr1}, \eq{pr2}, \eq{J}. 

Abnormal extremal trajectories are one-parameter subgroups $g_t$ tangent to the distribution $\D$:
\be{abnormal}
\dot g = u_1 X_1 + u_2 X_2, \qquad u_1, \ u_2 \equiv \const.
\ee
They project to straight lines in the plane $(x,y)$,  thus are optimal. 

 Normal extremals satisfy the Hamiltonian system
\be{norm_ham}
\dot{\lambda}= \vec{H}(\lambda), \qquad \lambda \in T^* G,
\ee 
where $H = \frac{1}{2}\left(h_1^2+h_2^2\right)$, $h_i(\lam) = \langle\lam, X_i\rangle$.  
In coordinates $(h_1, \dots, h_5; g)$ on $T^*G$, this system reads as
\begin{align}
&\dh_1 = -h_2 h_3, \label{dh1}\\
&\dh_2 = h_1 h_3, \label{dh2}\\
&\dh_3 = h_1 h_4 + h_2 h_5, \label{dh3}\\
&\dh_4 = 0, \label{dh4}\\
&\dh_5 = 0, \label{dh5}\\
&\dot g = h_1 X_1 + h_2 X_2. \label{dq_vH}
\end{align}
Normal extremal controls are $u_1 = h_1$, $u_2 = h_2$.

Since the vertical part \eq{dh1}--\eq{dh5} of this system is homogeneous, we can restrict to the level surface $\{ H = 1/2\}$, which corresponds to length parameterization of extremal trajectories. 

Abnormal extremal trajectories \eq{abnormal} are simultaneously normal.

\subsubsection{Integrability}

There are 3 independent Casimir functions~\cite{ABB} on the Lie coalgebra $\g^*$: the Hamiltonians $h_4$, $h_5$, and 
$$
E = \frac{h_3^2}{2} + h_1 h_5 - h_2 h_4.
$$
 Symplectic leaves of maximal dimension are two-dimensional:
\be{leaf2}
\text{connected components of }
\{p \in \g^* \mid h_4, \ h_5, \ E \equiv \const\},
\ee
thus the vertical subsystem~\eq{dh1}--\eq{dh5} is Liouville integrable. 
The symplectic foliation on $\g^*$ consists of:
\begin{itemize}
\item
two-dimensional parabolic cylinders~\eq{leaf2} in the domain $\{h_4^2+h_5^2 \neq 0\}$,
\item
two-dimensional affine planes $\{h_4 = h_5 = 0, \ h_3 = \const \neq 0\}$, and
\item 
points $(h_1, h_2) = \const$ in the plane $\{h_3^2 + h_4^2+h_5^2 = 0\}$.
\end{itemize}

\begin{remark}
For free nilpotent Lie coalgebras of $\text{rank } \geq 3$ and $\text{step } \geq 3$, and   of $\text{rank } \geq 2$ and $\text{step } \geq 4$, typical symplectic leaves have dimension at least $4$, and the corresponding Hamiltonian systems are not Liouville integrable~\cite{borisov, LS}.
\end{remark}

Introduce coordinates $(\theta, 
c, \alpha, \beta)$ on the level surface $\left\{\lam \in T^* M \mid H=\frac{1}{2}\right\}$ by the following formulas:
\be{h1245}
h_1 = \cos \t, \  h_2 = \sin \t, \quad  c = h_3, \quad h_4 = \a \sin \b, \ h_5 = -\a \cos \b,
\ee
then $E = \frac{c^2}{2} - \a \cos(\theta- \beta)$.

 The cylinder 
$$
C = \{ \lam \in \g^* \mid H(\lam) = 1/2 \}
$$ 
has the following stratification depending on its intersection with the symplectic leaves:
\begin{align*}
&C = \bigsqcup_{i=1}^7 C_i, \qquad C_i \cap C_i = \emptyset, \ i \neq j, \\
&C_1 = \{ \lam \in C \mid \a > 0, \ E \in (-\a, \a) \}, \\
&C_2 = \{ \lam \in C \mid \a > 0, \ E \in (\a,  + \infty) \}, \\
&C_3 = \{ \lam \in C \mid \a > 0, \ E =  \a , \ \t - \b \neq \pi  \}, \\
&C_4 = \{ \lam \in C \mid \a > 0, \ E  = -\a  \}, \\
&C_5 = \{ \lam \in C \mid \a > 0, \ E =  \a, \ \t - \b = \pi \}, \\
&C_6 = \{ \lam \in C \mid \a =  0, \ c \neq 0 \}, \\
&C_7 = \{ \lam \in C \mid \a = c = 0 \}.
\end{align*}

Intersections of symplectic leaves with the level surface of the Hamiltonian $\{H = 1/2\}$ are shown in Figs.~\ref{fig:cartan1}--\ref{fig:cartan6}.

\twofiglabelsize{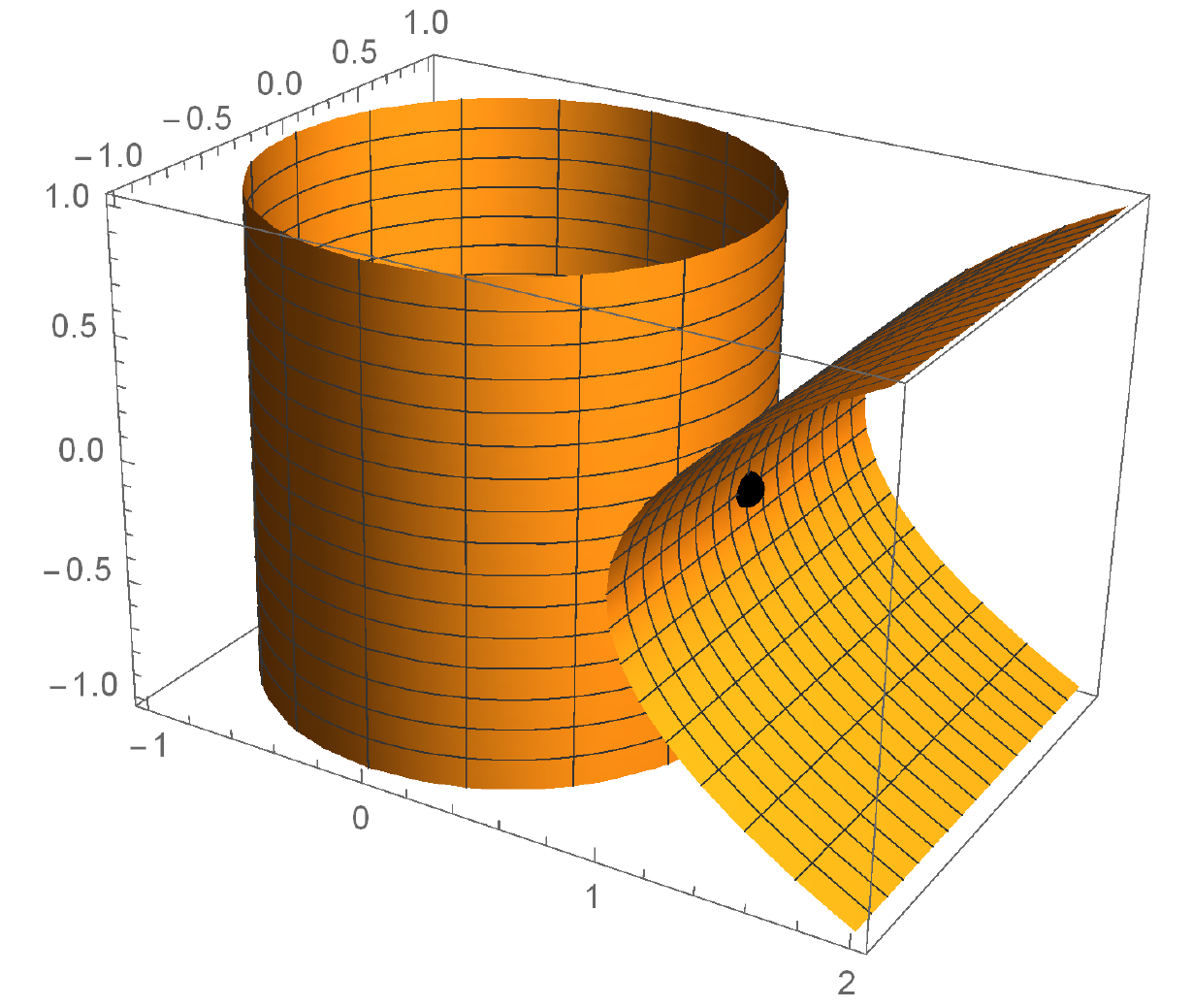}{$E = - \a> 0$: $\lam \in C_4$}{fig:cartan1}{5}{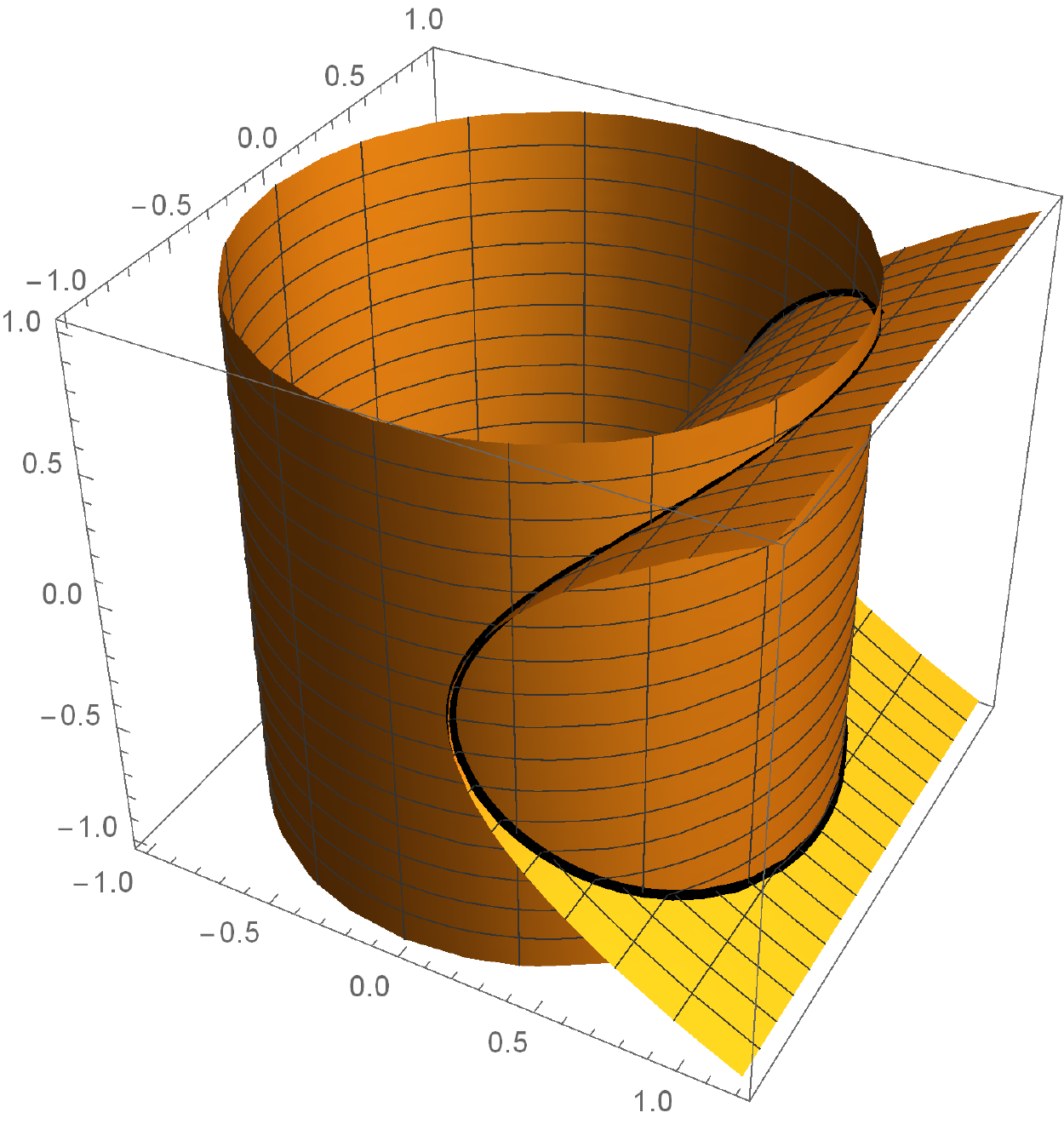}{$E \in (-\a, \a)$, $\a > 0$: $\lam \in  C_1$}{fig:cartan2}{5}

\twofiglabelsize{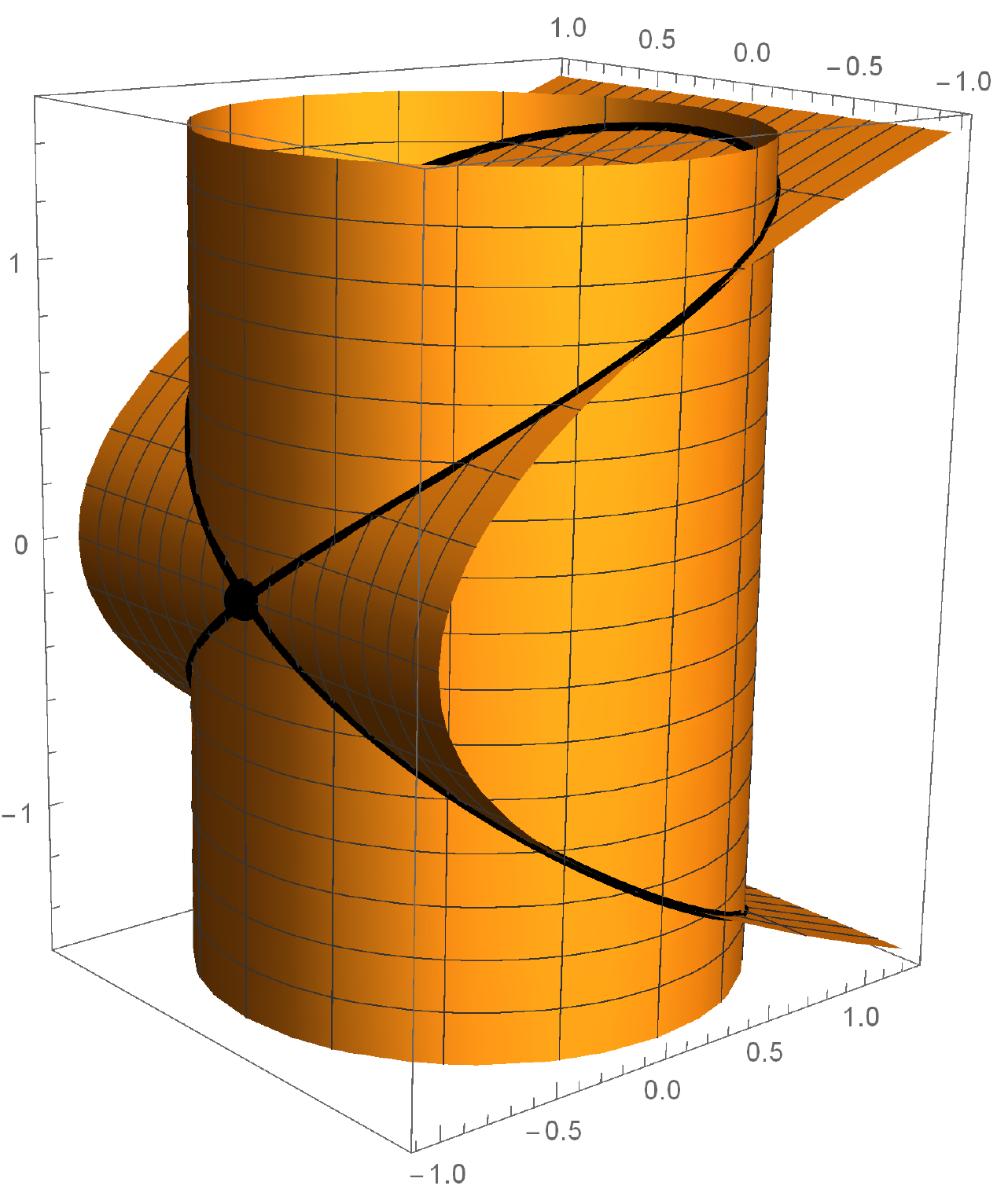}{$E = \a> 0$: $\lam \in  C_3 \cup C_5$}{fig:cartan3}{6}{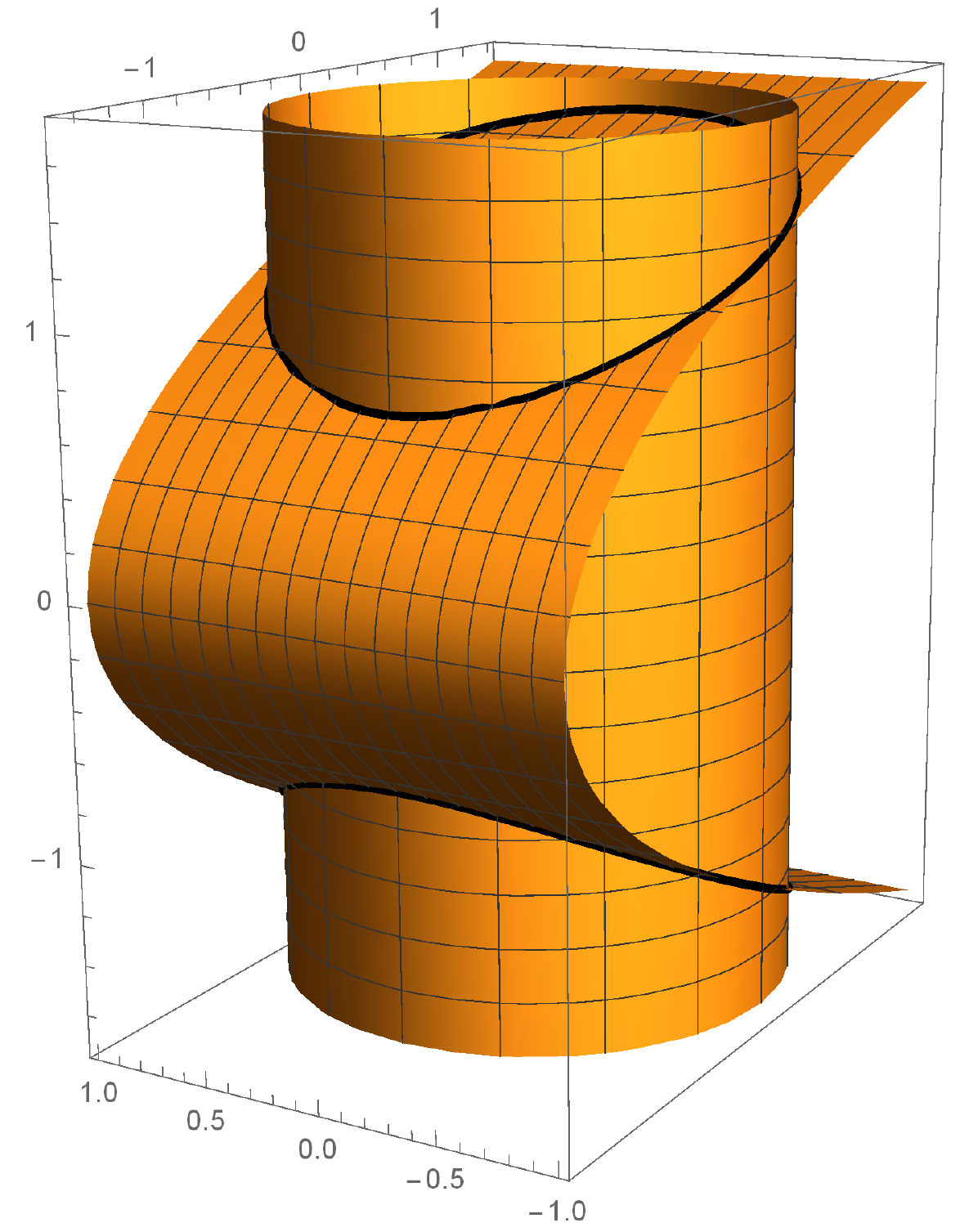}{$E > \a > 0$: $\lam \in  C_2$}{fig:cartan4}{6}

\twofiglabelsize{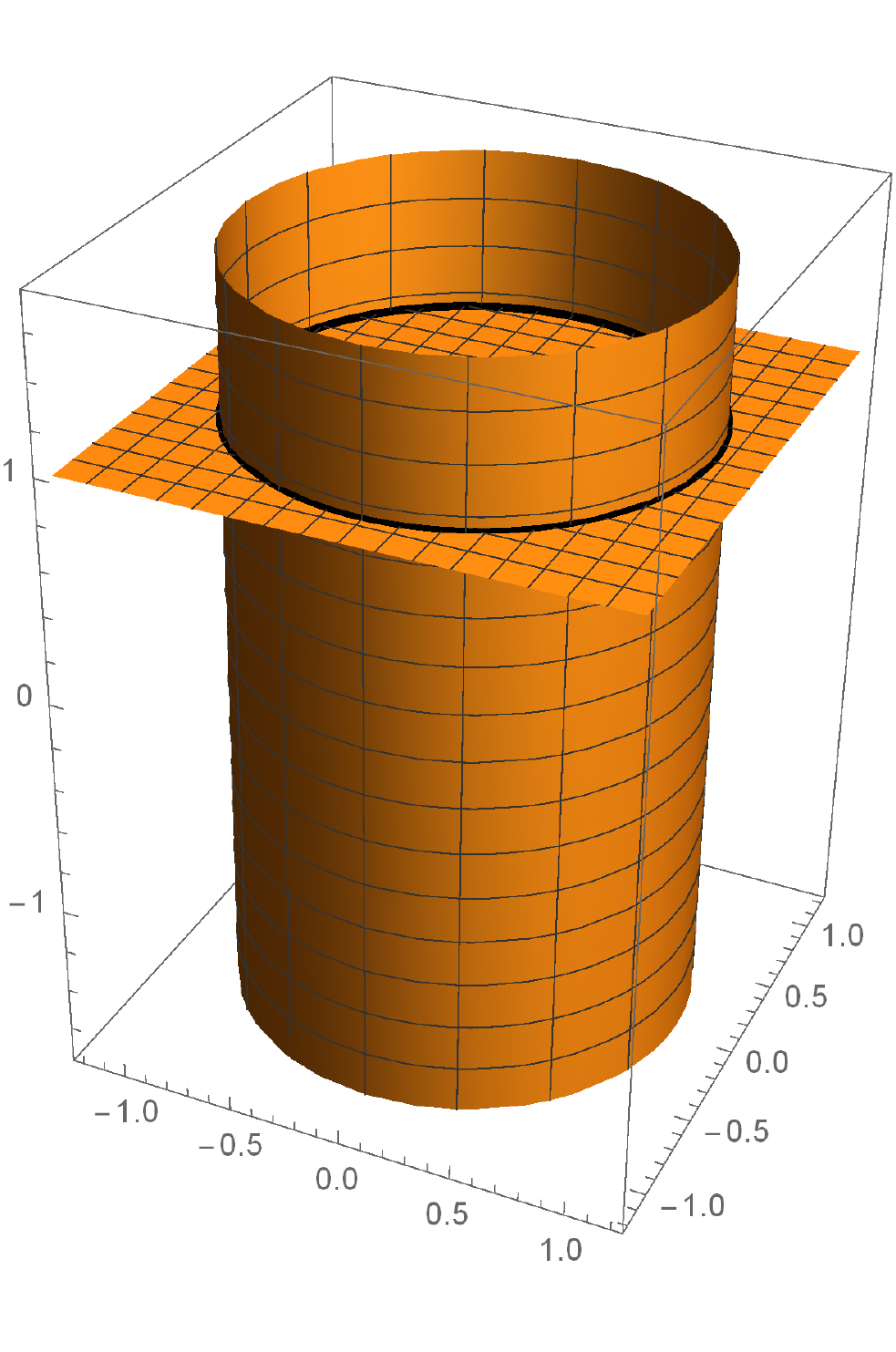}{$\a= 0$, $c \neq 0$: $\lam \in  C_6$}{fig:cartan5}{6}{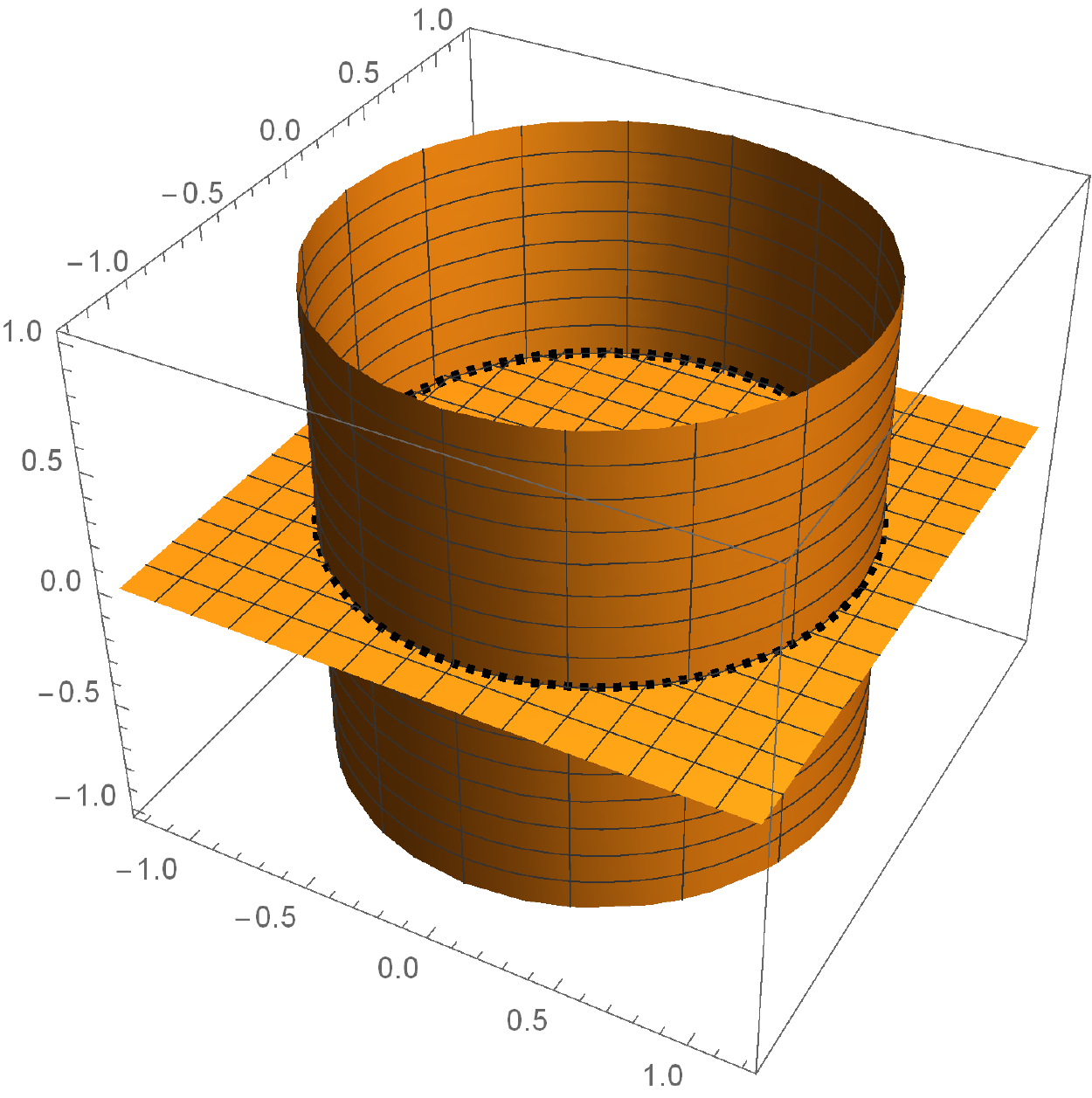}{$\a = c =  0$: $\lam \in  C_7$}{fig:cartan6}{4}

It is obvious from Figs.~\ref{fig:cartan1}--\ref{fig:cartan6} that extremal control $u(t) = (u_1(t), u_2(t)) = (h_1(t), h_2(t))$ has the following nature:
\begin{itemize}
\item
$u(t) \equiv \const$ in the cases $\lam \in C_4 \cup C_5 \cup C_7$, 
\item
$u(t)$ is periodic in the cases $\lam \in C_1 \cup C_2 \cup C_6$, 
\item
 $u(t)$ is nonperiodic but asymptotically constant (has finite limits as $t \to \pm \infty$) in the case $\lam \in C_3$.
\end{itemize}

\subsubsection{Continuous symmetries}
The Lie algebra of infinitesimal symmetries of the distribution $\D$ was described by E.~Cartan~\cite{cartan}: it is the 14-dimensional Lie algebra $\g_2$ --- the unique noncompact real form of the complex exceptional simple Lie algebra $\g_2^{\mathbb C}$~\cite{postnikov}.  See a modern exposition and explicit construction of these symmetries in~\cite{sym}.

The Lie algebra of infinitesimal symmetries of the SR structure $(\D, \langle \, \cdot \, , \, \cdot \, \rangle)$ is 6-dimensional: it contains 5 basis right-invariant vector fields on $G$ plus a vector field
\begin{align*}
&X_0 = - y \pder{}{x} + x \pder{}{y} - w \pder{}{v} + v \pder{}{w},\\
&[X_0, X_1] = - X_2, \qquad [X_0, X_2] = X_1,
\end{align*}
its flow is a simultaneous rotation in the planes $(x,y)$ and $(v,w)$:
\begin{multline*}
e^{sX_0}(x,y,z, v, w) \\
= (x \cos s - y \sin s, x \sin s + y \cos s , z, v \cos s - w \sin s, v \sin s + w \cos s).
\end{multline*}

There exists also a vector field $Y \in \Vec G$, an infinitesimal symmetry of $\Delta$,  such that
$$
[Y, X_1] = - X_1, \qquad [Y, X_2] = - X_2, \qquad [Y, X_0] = 0,
$$
this is the vector field
$$
Y = x \pder{}{x} + y \pder{}{y}+ 2z \pder{}{z}+3v\pder{}{v}+3w\pder{}{w},
$$
its flow is given by dilations:
$$
e^{rY} (x,y,z,v,w) = (e^r x, e^ry, e^{2r}z, e^{3r}v, e^{3r}w).
$$

Introduce the Hamiltonians $h_0(\lam) = \langle \lam, X_0(g)\rangle$, $h_Y(\lam) = \langle \lam, Y(g)\rangle$, $\lam \in T^*G$, the corresponding Hamiltonian vector fields $\vh_0, \vh_Y \in \Vec(T^* G)$, and the vector field $Z = \vh_Y + \sum_{i=1}^5 h_i \pder{}{h_i} \in \Vec(T^* G)$. Then the rotation $\vh_0$ is a symmetry of the normal Hamiltonian vector field:
\begin{align*}
&[\vh_0, \vH] = 0, \qquad \vh_0 H = 0, \\
&e^{s\vh_0} \circ e^{t\vH}(\lam) = e^{t\vH} \circ e^{s\vh_0}(\lam), \qquad \lam \in T^*G, \quad s, t \in \R,
\end{align*}
and the dilation $Z$ is a generalized symmetry of $\vH$: 
\begin{align*}
&[Z, \vH] = -\vH, \qquad Z H = 0, \\
&e^{rZ} \circ e^{t\vH}(\lam) = e^{t'\vH} \circ e^{rZ}(\lam), \qquad t' = t e^r, \qquad \lam \in T^*G, \quad r, t \in \R.
\end{align*}
Thus
\begin{align}
&e^{sX_0} \circ \Exp(\lam, t) = \Exp(e^{s\vh_0}(\lam), t), \label{X0h0}\\
&e^{rY} \circ \Exp(\lam, t) = \Exp(e^{rZ}(\lam), t'), \qquad t' = t e^r. \label{YZ}
\end{align}

\subsubsection{Pendulum and elasticae}

On the level surface $\{\lam \in T^*G \mid H(\lam) = 1/2\}$ the vertical subsystem~\eq{dh1}--\eq{dh5} of 
the normal Hamiltonian system takes in coordinates~\eq{h1245} the form of a {\em generalized pendulum}:
\be{pend}
\ddot{\t} = - \a \sin(\t - \b), \quad \a, \ \b = \const.
\end{equation}

See  the phase portrait of pendulum~\eq{pend} at Fig.~\ref{fig:cartan7} and Fig.~\ref{fig:cartan8} for the cases $\a > 0$ and $\a = 0$ respectively.

\twofiglabelsize{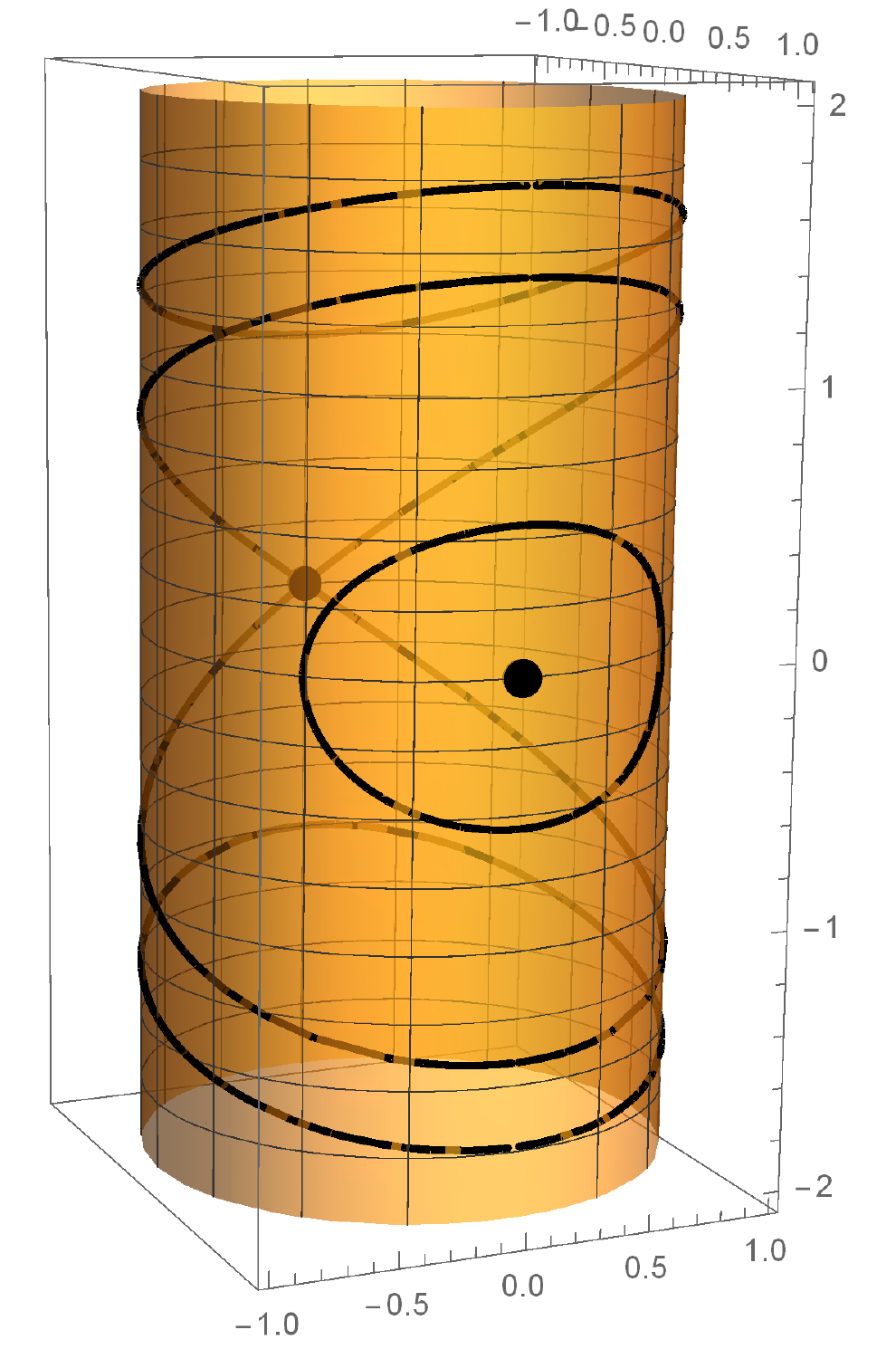}{Phase portrait of pendulum \eq{pend}, $\a >  0$}{fig:cartan7}{7}{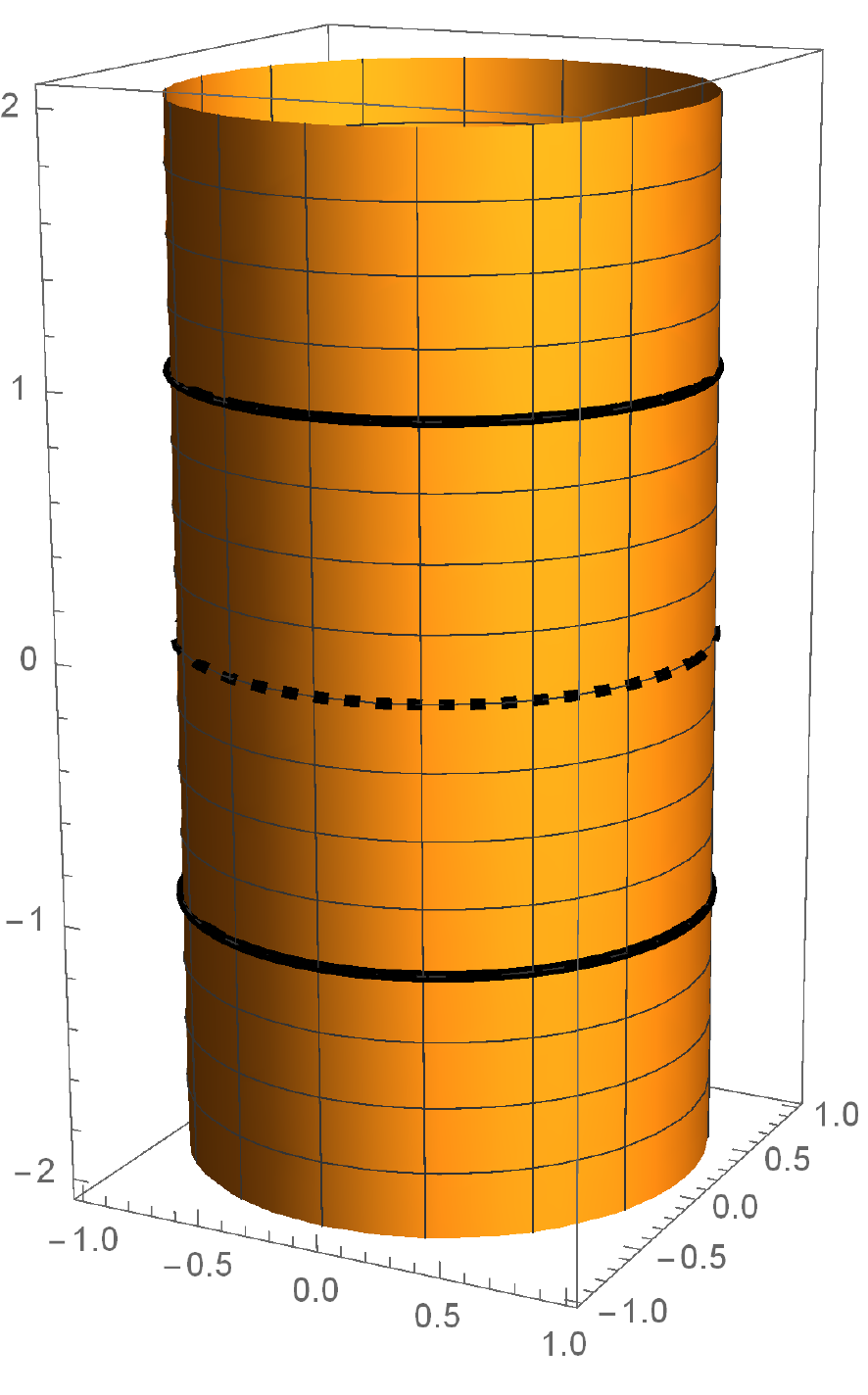}{Phase portrait of pendulum \eq{pend}, $\a = 0$}{fig:cartan8}{7}

Projections of extremal trajectories to the plane $(x,y)$ satisfy the ODEs
$$
\dot x = \cos \t, \quad \dot y = \sin \t.
$$
Thus they are Euler elasticae~\cite{euler, love} --- stationary configurations of elastic rod in the plane. Elasticae have different shapes depending on different types of motion of pendulum~\eq{pend}:
\begin{itemize}
\item
for pendulum at the stable equilibrium with the minimal energy $E = - \a> 0$ ($\lam \in C_4$), elasticae are straight lines,
\item
for oscillating pendulum with low energy $E \in (-\a, \a)$, $\a > 0$ ($\lam \in C_1$), elasticae are periodic curves with inflection points,
\item
for pendulum with critical energy $E = \a>0$ ($\lam \in C_3 \cup C_5$), elasticae are straight lines or critical non-periodic curves, 
\item
for rotating pendulum with high energy $E > \a>0$ ($\lam \in C_2$), elasticae are periodic curves without inflection points,
\item
for  pendulum uniformly rotating without gravity ($\a = 0$, $c \neq 0$, $\lam \in C_6$), elasticae are circles,
\item
and for stationary pendulum without gravity $(\a = c = 0$, $\lam \in C_7$) elasticae are straight lines.
\end{itemize}
See the plots of Euler elasticae in~\cite{euler, jurd_book, dido_exp}. The pendulum is Kirchhoff's kinetic analog of elasticae ~\cite{jurd_book}.

\subsubsection{Exponential mapping}
The family of all normal extremals is parametrized by points of the phase cylinder of the pendulum
\begin{eqnarray*}
C = \left\{\lambda \in \g^*  \mid H(\lambda)= \frac{1}{2}\right\}= \left\{(\theta, c, \alpha,\beta  ) \mid \theta \in S^1, \ c \in \R, 
\ \a \geq 0, \ \beta \in S^1  \right\},
\end{eqnarray*}
and is given by the {\em exponential mapping}
\begin{eqnarray*}
&&\map{\Exp}{C \times \R_+}{G},\\
&&\Exp (\lambda, t) = g_t = (x_t, y_t, z_t, v_t,w_t).
\end{eqnarray*}

\subsubsection{Discrete symmetries of the exponential mapping}\label{subsubsec:disc_sym}
The quotient of the generalized pendulum~\eq{pend} modulo rotation $X_0$ and dilation $Y$ is the standard pendulum
\be{pend_st}
\dot \t = c, \quad \dot c = - \sin \t, \qquad (\t, c) \in S^1 \times \R.
\ee
The field of directions of this equation has obvious  discrete symmetries --- reflections in the coordinate axes and in the origin
\begin{align*}
&\mapto{\eps^1}{(\t, c)}{(\t, - c)},\\
&\mapto{\eps^2}{(\t, c)}{(-\t, c)},\\
&\mapto{\eps^3}{(\t, c)}{(-\t, - c)}.
\end{align*}
These reflections generate a dihedral group
$$
D_2 = \{\Id, \eps^1, \eps^2, \eps^3\} = \Z_2 \times \Z_2.
$$
Action of reflections continues naturally to Euler elasticae $(x_t, y_t)$, so that modulo rotations in the plane $(x,y)$:
\begin{itemize}
\item
$\eps^1$ is the reflection of an elastica in the center of its chord,
\item
$\eps^2$ is the reflection of an elastica in the middle perpendicular to its chord,
\item
$\eps^3$ is the reflection of an elastica in its chord.
\end{itemize}
Further, action of reflections naturally continues to preimage of the exponential mapping:
$$
\map{\eps^i}{C \times\R_+}{C \times \R_+}, \qquad i = 1, 2, 3, 
$$
and to its image:
$$
\map{\d^i}{G}{G}, \qquad i = 1, 2, 3, 
$$
so that
\be{epsdel}
\d^i \circ \Exp(\lam, t) = \Exp \circ \eps^i(\lam,t), \qquad (\lam, t) \in C \times \R_+, \quad i = 1, 2, 3,
\ee
and $\eps^i$ preserves time: $\eps^i(\lam, t) = (*, t)$. In such a case we say that the pair of mappings $(\eps^i, \d^i)$ is a {\em symmetry of the exponential mapping}.

Combining reflections and rotations, we obtain a group $\Sym$ of symmetries of the exponential mapping:
\begin{align}
&e^{s \vh_0}, \quad \map{e^{s\vh_0} \circ \eps^i}{C \times \R_+}{C \times \R_+}, \label{Sym1}\\ 
&e^{s X_0}, \quad \map{e^{sX_0} \circ \d^i}{G}{G}. \label{Sym2}
\end{align}
 
 Notice that
\be{eps3}
\eps^3(\t, c, \a, \b, t) = (-\t, -c, \a, -\b, t), \qquad \lam = (\t, c, \a, \b) \in C, \quad t \in \R_+,
\ee
i.e., the action of $\eps^3$ does not depend on $t$.
\subsubsection{Integration of the normal Hamiltonian system}
The equation of pendulum~\eq{pend} is integrable in Jacobi's functions, thus the normal Hamiltonian system~\eq{norm_ham} is integrable in Jacobi's functions as well.

In order to parameterize extremal trajectories explicitly, in work~\cite{dido_exp} were introduced elliptic coordinates  
 $(\varphi, k, \alpha, \b)$ on the sets $C_1$, $C_2$, $C_3$  in the following way:
\begin{itemize}
\item
$\f \in \R$ is the time of motion of pendulum~\eq{pend} from a chosen initial curve to a current point, and
\item
$k$ is a reparameterized energy $E$:
\begin{align*}
&\lam \in C_1 \quad \then \quad k = \sqrt{\frac{E+\a}{2\a}} \in (0,1), \\
&\lam \in C_2 \quad \then \quad k = \sqrt{\frac{2\a}{E+\a}} \in (0,1), \\
&\lam \in C_3 \quad \then \quad k = 1.
\end{align*}
\end{itemize}

In the elliptic coordinates
$(\f, k, \a, \b)$ on 
$\cup_{i=1}^3 C_i$ the vertical part of the Hamiltonian system~\eq{pend}
rectifies:
$$
\dot{\f} = 1, \quad \dot{k} = \dot{\a} = \dot{\b} = 0.
$$

In~\cite{dido_exp}
these coordinates were used  for explicit parameterization of extremal trajectories in terms of Jacobi functions $\sn (u,k)$, $\cn (u,k)$, $\dn (u,k)$, $\E(u,k) = \int_0^u \dn^2(t,k) \, dt$~\cite{whit_vatson}.

\subsubsection{Optimality of normal extremal trajectories}
Short arcs of normal extremal trajectories are (globally) optimal, thus they are SR geodesics. But long arcs of geodesics are, in general, not optimal. The instant at which a geodesic loses its optimality is called the {\em cut time}:
$$\tcut(\lambda) = \sup \{ t>0 \mid \mathrm {Exp} (\lambda, s) \text{ is optimal for } s \in [0,t]\}.$$

A geodesic is called {\em locally optimal} if it is optimal w.r.t. all trajectories with the same endpoints, in some neighborhood in the topology of $C([0, t_1], G)$ (or, which is equivalent, in the topology of $G$). The instant when a geodesic loses its local optimality is called the {\em first conjugate time}:
$$
\tconj(\lam) = 
\sup \{ t>0 \mid \mathrm {Exp} (\lambda, s) \text{ is locally optimal for } s \in [0,t]\}.$$
There are 3 reasons for the loss of global optimality for SR geodesics~\cite{ABB}:
\begin{enumerate}
\item
intersection points of different geodesics of equal length (such points are called {\em Maxwell points}),
\item
conjugate points,
\item
abnormal trajectories.
\end{enumerate}

As we mentioned above, all abnormal trajectories in the SR problem on the Cartan group are optimal.
Maxwell points in this problem were studied in detail in~\cite{max1, max2, max3}. A typical reason for Maxwell points is a symmetry of the exponential mapping. Along each geodesic $g_t = \Exp(\lam, t)$, $\lam \in C$, was found the first Maxwell time corresponding to the group of symmetries $\Sym$~\eq{Sym1}, \eq{Sym2}. The {\em first Maxwell time} 
$$
\map{\tmax}{C}{(0, + \infty]}
$$
corresponding to the group $\Sym$ was described. Thus the following upper bound of the cut time was proved:

\begin{theorem}[\cite{max3}, Theorem 6.1]\label{th:tcut_bound}
For any $\lambda \in C$ 
\begin{align}
\tcut(\lambda) \leq \tmax(\lambda). \label{tcutbound}
\end{align}
\end{theorem}

The first Maxwell time corresponding to the group of symmetries $\Sym$ is explicitly defined as follows:
\begin{align}
&\lam \in C_1 \then \tmax(\lam) = 
\min \left(\frac{2}{\sqrt{\a}}p_1^z(k), 
\frac{2}{\sqrt{\a}}p_1^V(k)\right),  \label{p1C1}\\
&\lam \in C_2 \then \tmax(\lam) =  
\frac{2k}{\sqrt{\a}}p_1^V(k), \nonumber \\
&\lam \in C_6 \then \tmax(\lam) =  \frac{4}{|c|}p_1^V(0),  \nonumber\\
&\lam \in C_i, \ i = 3, 4, 5, 7 \then \tmax(\lam) =  + \infty. \nonumber
\end{align}
Here $p_1^z = p_1^z(k) \in (K, 3 K)$ is the first positive root of the equation $f_z(p,k) =  0$, where
\be{fzC1}
f_z(p,k) = \sn p \dn p - (2 \E(p) -  p) \cn p,
\ee
$p_1^V = p_1^V(k)$ is the first positive root of the equation $f_V(p,k) =  0$, where
\begin{align}
f_V(p) &= 
\dfrac 43 \, \ss \, \dd \, (-p - 2(1 - 2 k^2 + 6 k^2 \cc^2)(2 \E(p) - p) + (2 \E(p) - p)^3 \nonumber\\
&\qquad + 8 k^2  \cc \, \ss \, \dd) +   4  \cc \, (1 - 2 k^2  \ss^2)(2 \E(p) - p)^2, \quad  \label{fVC1}\\
p_1^V(k) &\in [2K, 4 K) \qquad \text{for } \lam \in C_1, \nonumber
\intertext{and}
f_V(p) &= 
\dfrac 43 \{ 3\,\dd\,(2\E(p) - (2 - k^2)p)^2   + \cc\,
[8\E^3(p) - 4\E(p)(4 + k^2) \nonumber\\
&\qquad -
 12\E^2(p)(2-k^2)p 
 + 6\E(p)(2-k^2)^2 p^2 \nonumber\\
&\qquad + 
 p(16 - 4k^2-3k^4-(2-k^2)^3p^2)]\,\ss -  
 \nonumber \\
 &\qquad -
2\,\dd\,(-4k^2+3(2\E(p)-(2-k^2)p)^2)\,\ss^2 + 
\nonumber\\
& \qquad 
+ 12 k^2 \cc(2 \E(p) - (2 - k^2)p)\ss^3 
- 8 k^2\,\ss^4\dd\},
\quad  \label{fVC2} \\
p_1^V(k) &\in (K, 2 K) \qquad \text{for } \lam \in C_2, \nonumber
\end{align}
and $p_1^V(0) \in (\pi/2, \pi)$ is the first positive root of the function
$$
f_V^0(p) = [(32 p^2 - 1) \cos 2 p - 8 p \sin 2 p + \cos 6 p]/512.
$$
Here $K = K(k)$ is the complete elliptic integral of the first kind.

\subsection{Result of this paper}\label{subsec:result}

In this article we study local optimality of SR geodesics on the Cartan group and estimate the first conjugate time. 
The main result is  the following  bound.

\begin{theorem}\label{th:tconj}
For any $\lambda \in C$
\be{tconjmax}
\tconj(\lambda) \geq \tmax(\lambda).
\ee 
\end{theorem}

In Sections~\ref{sec:C1}--\ref{sec:C47} we prove inequality~\eq{tconjmax}, $\lambda \in C_i$ for all $i=1,\dots,7$.
Theorem~\ref{th:tconj} follows immediately from Theorems~\ref{th:tconjC1}, \ref{th:tconjC2}, \ref{th:tconjC47}, \ref{th:tconjC3}, \ref{th:tconjC6}.

\medskip

Theorem~\ref{th:tconj} is interesting from two points of view. First, it describes locally optimal geodesics in a sharp way; we will see  in Sec.~\ref{sec:tconj=} that for certain geodesics inequality~\eq{tconjmax} turns into equality. Second, Th.~\ref{th:tconj} is an important step in the study of global optimality in the SR problem on the Cartan group. Namely, we conjectured in~\cite{max3} that inequality~\eq{tcutbound} is in fact equality:

\begin{conjecture}
For any $\lam \in C$
$$
\tcut(\lam) = \tmax(\lam).
$$
\end{conjecture}
In a forthcoming paper~\cite{ard_hak}, this conjecture is proved with the use of Th.~\ref{th:tconj} via a symmetry method~\cite{ABB, engel_cut, cut_sre1, cut_sh2}.

\subsection{Methods of this paper}

Consider the   Jacobian of the exponential map 
$$
J_0 = 
\ds\frac{\partial (x, y, z, v, w)}{\partial (\theta, c, \alpha, \beta,  t)} =
\begin{array}{|c c c|}
\pder{x}{\theta} & \ldots & \pder{x}{t} \\
\vdots & \ddots & \vdots \\
\pder{w}{\theta} & \ldots & \pder{w}{t}
\end{array} \, .
$$
A point $g_t = \Exp(\lam,t)$ on a strictly normal geodesic is  {\em conjugate} iff $ (\lam,t)$ is a critical point of the exponential mapping, that is why $g_t$ is the corresponding critical value:
$J_0(\lam,t) = 0$.  So inequality~\eq{tconjmax} is equivalent to the following one:
\be{J0}
J_0(\lam, t) \neq 0 \text{ for } t \in (0, \tmax(\lam)).
\ee
This inequality is proved in Secs.~\ref{sec:C1}--\ref{sec:C47} for 
$\lam \in C =  \cup_{i=1}^7  C_i$ as follows:
\begin{enumerate}
\item
First generic cases $\lam\in C_1$ and $\lam \in C_2$ are considered directly and independently.
\item
The Jacobian of the exponential mapping is computed w.r.t. the coordinates $(\f, k, \a, \b)$.
\item
On the basis of continuous symmetries (rotations  and dilations), the $5 \times 5$ Jacobian $\ds \pder{(x,y,z,v,w)}{(\f, k, \a, \b, t)}$ is reduced to a $3 \times 3$ Jacobian $\ds \pder{(P, Q, R)}{(\f, k, t)}$, where $P, Q, R$ are invariants of the continuous symmetries.
\item
The Jacobian  $\ds \pder{(P, Q, R)}{(\f, k, t)}$ is computed explicitly via parameterization of the exponential mapping, and is simplified to a function $J_1$.
\item
We prove the inequality
$$
J_1 \neq 0 \text{ for } t \in (0, \tmax(\lam))
$$
basically by three methods which we explain below:
\begin{enumerate}
\item
homotopy invariance of the index of the second variation, 
\item
method of comparison function,
\item
``{\em divide et impera}'' method.
\end{enumerate}
\item
In such a way, inequality~\eq{J0} is proved for the generic cases $\lam \in C_1$ and $\lam \in C_2$ (Secs.~\ref{sec:C1} and \ref{sec:C2} respectively).
\item
For $\lam \in C_4 \cup C_5 \cup C_7$, the geodesics are globally optimal since they project to straight lines in the plane of the distribution, and the required bound follows trivially.
\item
For $\lam \in C_3$, the geodesics are globally optimal since they project to length minimizers on the Engel group~\cite{engel_cut}, and the required bound follows.
\item
Bound \eq{J0} is proved for the special case $\lam \in C_6$ by a limit passage from the generic case $\lam \in C_2$ via the homotopy invariance of the index of the second variation.
\item
We conclude that bound \eq{J0} is proved for all $\lam \in C$.
\end{enumerate}

\paragraph{Homotopy invariance of index of second variation.}
Under certain nondegeneracy conditions (see Sec.~\ref{sec:conj_hom}), for each strictly normal geodesic one can define the index of the second variation equal to the number of conjugate points with account of their multiplicity. A geodesic does not contain conjugate points iff its index vanishes. The index of the second variation is preserved under homotopies of geodesics such that their endpoints are not conjugate. So we can prove absence of conjugate points on a geodesic if we construct its homotopy to a ``simple'' geodesic without conjugate points, in such a way that endpoints of all geodesics in the continuous family are not conjugate. In Sec.~\ref{sec:C2} we construct such a homotopy from a geodesic $g_t = \Exp(\lam, t)$, $\lam = (\f, k, \a, \b) \in C_2$, to a geodesic $\bar g_t = \Exp(\bar \lam, t)$, $\bar \lam = (\bar \f, \bk, \a, \b) \in C_2$, where $\bk$ is close to $0$. The geodesic $\bar g_t$ is ``simple'' because for $\bar k \to 0$ the exponential mapping is asymptotically expressed by trigonometric functions, not Jacobi ones as for general $k \in (0, 1)$. See details on the homotopy invariance of the index of the second variation in Sec~\ref{sec:conj_hom}.

\paragraph{Comparison functions.}
Let $f_0(t)$, $f_1(t)$ be real-analytic functions on an interval $t \in (a,b)$. The function $f_1(t)$ is called a {\em comparison function} for $f_0(t)$ if
\begin{align}
&\left( \frac{f_0(t)}{f_1(t)}\right)' \cdot f_1^2(t) \geq 0 \quad (\text{or } \leq 0),  \label{der}\\
&f_1(t) \neq 0, \qquad t \in (a, b), \nonumber
\end{align}
and equality in \eq{der} is possible only for isolated values of $t$.
Then $\ds \frac{f_0(t)}{f_1(t)}$ increases (or decreases) for $t \in (a,b)$. 

In such a way one can often bound  $\ds\frac{f_0(t)}{f_1(t)}$, and after that bound $f_0(t)$ in the case when it is not monotone.

\paragraph{``{\em Divide et impera}''.}
For a given function $f_0(t)$, we find a sequence of functions $f_1(t)$, $f_2(t)$, \dots, $f_N(t)$ such that $f_i$ is a comparison function for $f_{i-1}$, $i = 1, \dots, N$.
In such a way
we divide the analytical complexity of the function~$f_0(t)$ into several parts and bound the obtained parts step by step. Moreover, at each step we divide the function under study by an appropriate divisor. 

The method works perfectly for trigonometric quasipolynomials as follows. Consider a trigonometric quasipolynomial 
$$
f_0(t) =  t^N \tf_0(t) + o(t^N), \qquad t \to \infty,
$$
where $\tf_0(t)$ is a trigonometric function. Compute successfully:
\begin{align*}
&f_1(t) = \left(\frac{f_0}{\tf_0}\right)' \cdot \tf_0^2 =  t^{N-1} \tf_1(t) + o(t^{N-1}), \qquad t \to \infty,\\
&\cdots, \\
&f_{N-1}(t) = \left(\frac{f_{N-2}}{\tf_{N-2}}\right)' \cdot \tf_{N-2}^2 =  t \tf_{N-1}(t) + o(t), \qquad t \to \infty,\\
&f_N(t) = \left(\frac{f_{N-1}}{\tf_{N-1}}\right)' \cdot \tf_{N-1}^2,
\end{align*}
and $f_N(t)$ is a trigonometric function. 

We determine the sign of $f_N(t)$ on an interval $t \in(a,b)$ under investigation, and on this basis determine monotonicity of $\ds\frac{f_{N-1}}{\tf_{N-1}}$. Then we determine the signs of $\tf_{N-1}$ and $f_{N-1}$ and so on. In such a way we obtain successfully bounds for $f_{N-1}$, $f_{N-2}$, \dots, $f_0$.

\paragraph{Projection to lower-dimensional SR minimizers.}
There is a general construction of projecting left-invariant SR structures to quotient groups such that optimality of projected geodesics in the quotient group implies optimality of the initial geodesic.

Let $(\Delta, \langle \, \cdot \, , \, \cdot \,  \rangle)$ be a left-invariant SR structure on a Lie group $G$. Let $G_0 \subset G$ be a closed normal subgroup whose Lie algebra intersects trivially with $\Delta$. Consider the quotient $\map{\pi}{G}{H = G /G_0}$. Then $\pi_* \Delta \subset TH$  is a left-invariant distribution on $H$, and $\map{\pi_{*g}}{\Delta_g}{(\pi_* \Delta)_{\pi(g)}}$ is an isomorphism. Denote by $\pi_* \langle \, \cdot \, , \, \cdot \,  \rangle$ the inner product in $\pi_* \Delta$ induced by $\pi_*$. Then $(\pi_* \Delta, \pi_* \langle \, \cdot \, , \, \cdot \,  \rangle)$ is a left-invariant SR structure on $H$, of the same rank as  $(\Delta, \langle \, \cdot \, , \, \cdot \,  \rangle)$. If $g_t \in G$ is a horizontal curve for $\Delta$, then $\pi(g_t)$ is a horizontal curve for $\pi_* \Delta$. Moreover, if $\pi(g_t)$ is optimal for $(\pi_* \Delta, \pi_* \langle \, \cdot \, , \, \cdot \,  \rangle)$, then $g_t$ is optimal for $(\Delta, \langle \, \cdot \, , \, \cdot \,  \rangle)$.

 \subsection{Structure of this paper}
In Sec.~\ref{sec:sym} we establish invariance of the cut time and the first conjugate time w.r.t. symmetries of the problem (rotation $X_0$, dilation $Y$, and reflection~$\eps^3$).

In Sec.~\ref{sec:conj_hom} we recall the results on homotopy invariance of the index of the second variation in the form we require in Secs. \ref{sec:C2}, \ref{sec:C47}.

In Secs.~\ref{sec:C1} and  \ref{sec:C2} we prove Th.~\ref{th:tconj} for the generic cases $\lam \in C_1$ and $\lam \in C_2$ respectively. And in Sec.~\ref{sec:C47} we prove Th.~\ref{th:tconj} for the special cases $\lam \in \cup_{i=3}^7 C_i$.

In Sec. \ref{sec:tconj=} we present cases when the first conjugate time coincides with the first Maxwell time.
In Sec. \ref{sec:tconjcont} we show that the first conjugate time is continuous near infinite values, out of certain abnormal geodesics.

Finally, in Sec. \ref{tconj2side} we present a numerical evidence of two-sided bounds of the first conjugate time.

\subsection{Related works}
As we already mentioned, this work is an essential step towards construction of optimal synthesis and description of the cut time and cut locus for the left-invariant SR structure on the Cartan group. So far this goal was achieved just for few left-invariant SR structures on Lie groups. 

In the case  $\dim G = 3$, growth vector $(2,3)$, the following left-invariant SR structures were completely studied:  
\begin{itemize}
\item
Heisenberg group: A.M.~Vershik, V.Y.~Gershkovich \cite{versh_gersh}, R. Brockett~\cite{brock},
\item
axisymmetric metrics on  $\SO(3)$, $\SL(2)$:  U. Boscain and F. Rossi~\cite{boscainSO3},
V.N.~Berestovskii and  I.A. Zubareva \cite{berzub1, berzub2}, 
\item
$\SE(2)$, $\SH(2)$: Ya. Butt, A. Bhatti and the author \cite{cut_sre1,  cut_sh2}.
\end{itemize}

The free nilpotent SR structure with the  growth vector $(3,6)$ was studied by  O.Myasnichenko~\cite{myasnich}.

The SR problem on the Engel group (growth vector $(2, 3, 4)$) was studied by 
A.A.~Ardentov and the author~\cite{engel_synth}.

\section{Symmetries of cut time and   conjugate time}\label{sec:sym}
The exponential mapping is preserved by 3 symmetries:
rotations $e^{s \vh_0}$,
dilations~$e^{s Z}$,
reflection $\eps^3$.
Thus the cut time and the conjugate time are also preserved by these symmetries, see Cor.~\ref{cor:cutconjsym} below. This is proved via the following statement.

\begin{lemma}\label{lem:cutconjsym}
Let there exist homeomorphisms $\map{F}{C}{C}$, $\map{f}{G}{G}$ and a number $a>0$ such that
$$
f \circ \Exp(\lam, t) = \Exp(F(\lam), at), \qquad (\lam,t) \in C \times \R_+.
$$
Then $\tcut(F(\lam)) = \tcut(\lam)$ and $\tconj(F(\lam)) = \tconj(\lam)$ for all $\lam \in C$.
\end{lemma}
\begin{proof}
Let a geodesic $g_t = \Exp(\lam, t)$, $t \in [0, t_1]$, be optimal in a neighborhood $O \subset G$. We prove that the geodesic $\tg_t = \Exp(F(\lam), a t) = f(g_t)$, $t \in [0, t_1]$, is optimal in the neighborhood $f(O)$.

By contradiction, let there exist a geodesic better than $\tg_t$, $t \in [0, t_1]$:
\begin{align*}
&\{\hg_t = \Exp(\hlam, t) \mid t \in [0, \htt\,]\} \subset f(O), \\
&\hg_{\htt} = \tg_{t_1}, \qquad \htt < t_1.
\end{align*}
Consider then the geodesic
$$
\{ \tg_t = \Exp(F^{-1}(\hlam), t/a) = f^{-1}(\hg_t) \mid t \in [0, \htt\,]\} \subset O.
$$
We have
$$
\tg_{\htt} = f^{-1}(\hg_{\htt}) = f^{-1}(\tg_{t_1}) = f^{-1} \circ f(g_{t_1}) = g_{t_1},
$$
thus $\tg_{t}$, $t \in [0, \htt]$, is better than   $g_t$, $t \in [0, t_1]$, a 
contradiction.
\end{proof}

\begin{cor}\label{cor:cutconjsym}
For any $s \in \R$ and any $\lam \in C$ there hold the equalities
\begin{align*}
&\tcut \circ e^{s \vh_0}(\lam) = \tcut \circ e^{s Z}(\lam) = \tcut \circ \eps^{3}(\lam) = \tcut(\lam), \\  
&\tconj \circ e^{s \vh_0}(\lam) = \tconj \circ e^{s Z}(\lam) = \tconj \circ \eps^{3}(\lam) = \tconj(\lam).
\end{align*} 
\end{cor}
\begin{proof}
Apply Lemma~\ref{lem:cutconjsym} in the cases:
\begin{align*}
&F = e^{s \vh_0}, \quad f = e^{s X_0}, \quad a = 1,   &&\text{ see } (\ref{X0h0}), \\
&F = e^{s Z}, \quad f = e^{s Y}, \quad a = e^s,   &&\text{ see } (\ref{YZ}), \\
&F = \eps^{3}, \quad f = \delta^{3}, \quad a = 1,   &&\text{ see } (\ref{epsdel}), (\ref{eps3}).
\end{align*}
\end{proof}

\section{Conjugate points and homotopy}\label{sec:conj_hom}

In this section we recall some necessary facts from the theory of conjugate points in optimal control problems. 
We will need these facts in Secs.~\ref{sec:C2}, \ref{sec:C47}.
For details see, e.g., \cite{notes, cime, sar}.

Consider an optimal control problem of the form
\begin{align}
&\dot q = f(q,u), \qquad q \in M, \quad u \in U \subset \R^m, \label{sys} \\
&q(0) = q_0, \qquad q(t_1) = q_1, \qquad t_1 \text{ fixed}, \label{bound1} \\
&J = \int_0^{t_1} \f(q(t),u(t)) \, dt \to \min, \label{JJ}
\end{align}
where $M$ is a  finite-dimensional analytic manifold,  $f(q,u)$ and $\f(q,u)$ are respectively analytic in $(q,u)$  families of vector fields   and   functions  on $M$ depending on the control parameter $u \in U$, and $U$ an open subset of $\R^m$. Admissible controls are $u(\cdot) \in L_{\infty}([0, t_1],U)$, and admissible trajectories $q(\cdot)$ are Lipschitzian.
Let 
$$
h_u(\lam) = \langle \lam, f(q,u)\rangle - \f(q,u), 
\qquad \lam \in T^*M, \quad q = \pi(\lam) \in M, \quad u \in U,
$$
be the \ddef{normal Hamiltonian of PMP} for problem~\eq{sys}--\eq{JJ}. 

Fix a triple $(\widetilde{u}(t), \lam_t, q(t))$ consisting of a normal extremal control $\widetilde{u}(t)$, the corresponding extremal $\lam_t$, and a strictly normal extremal trajectory $q(t)$ for problem~\eq{sys}--\eq{JJ}.

Let the following hypotheses hold:

\hypoth{\Ho}{For all $\lam \in T^*M$ and  $u \in U$, the quadratic form $\ds \pdder{h_u}{u}(\lam)$ is negative definite.}

\hypoth{\Ht}
{For any $\lam \in T^* M$, the function $u \mapsto h_u(\lam)$, $u \in U$, has a maximum point $\bu(\lam) \in U$:
$$
h_{\bu(\lam)}(\lam) = \max_{u \in U} h_u(\lam), \qquad \lam \in T^*M.
$$}%

\hypoth{\Hth}
{The extremal control $\widetilde{u}(\cdot)$ is a corank one critical point of the endpoint mapping.}%

\hypoth{\Hf}
{ The Hamiltonian vector field $\vH(\lam)$, $\lam \in T^*M$, is forward complete, i.e., all its trajectories are defined for $t \in [0, +\infty)$. 
}

An instant $t_* > 0$ is called a \ddef{conjugate time} (for the initial instant $t = 0$) along the extremal $\lam_t$ if the restriction of the second variation of the endpoint mapping to the kernel of its first variation   is degenerate, see~\cite{notes} for details. In this case the point $q(t_*) = \pi(\lam_{t_*})$ is called \ddef{conjugate} for the initial point  $q_0$ along the extremal trajectory $q(\cdot)$.

Under hypotheses \Ho--\Hf, we have the following:

\begin{enumerate}
\item
Normal extremal trajectories lose their local optimality (both strong and weak) at the first conjugate point, see~\cite{notes}.
\item
An instant $t > 0$ is a conjugate time iff the exponential mapping $\Exp_t = \pi \circ e^{t \vH}$ is degenerate, see~\cite{cime}.
\item
Along each normal extremal trajectory, conjugate times are isolated one from another, see~\cite{sar}.
\end{enumerate}

We will apply the following statement for the proof of absence of conjugate points via homotopy.

\begin{theorem}[Corollary 2.2~\cite{el_conj}]\label{th:conj_hom}
Let $(u^s(t), \lam^s_t)$, $t \in [0, + \infty)$, $s \in [0, 1]$, be a continuous in parameter~$s$ family of normal extremal pairs in optimal control problem~\eq{sys}--\eq{JJ} satisfying hypotheses \Ho--\Hf. Let the corresponding extremal trajectories $q^s(t)$ be strictly normal.

Let $s \mapsto t_1^s$ be a continuous function, $s \in [0, 1]$, $t_1^s \in (0, + \infty)$.
Assume that for any $s \in [0, 1]$ the instant $t = t_1^s$ is not a conjugate time along the extremal~$\lam_t^s$.

If the extremal trajectory $q^0(t) = \pi(\lam_t^0)$, $t \in (0, t_1^0]$, does not contain conjugate points, then the extremal trajectory $q^1(t) = \pi(\lam_t^1)$, $t \in (0, t_1^1]$, also does not contain conjugate points.
\end{theorem}

One easily checks that the sub-Riemannian problem~\eq{pr1}, \eq{pr2}, \eq{J} satisfies all hypotheses \Ho--\Hf, so the results cited in this section are applicable to this problem.

\section{Conjugate time for $\lambda \in C_1$}\label{sec:C1}

In this section we prove Th. \ref{th:tconj} in the case $\lam \in C_1$:

\begin{theorem}\label{th:tconjC1}
If $\lam \in C_1$, then $\tconj(\lam) \geq \tmax(\lam)$.
\end{theorem}

Consider the Jacobian of the exponential mapping
$$
J = \frac{\partial(x, y, z, v, w)}{\partial(t, \f, k, \a, \b)}.
$$
We show that $J > 0$ for $t \in (0, \tmax(\lam))$, $\lam \in C_1$. 

\subsection{Transformation of Jacobian $J$}
As was shown in \cite{dido_exp}, the exponential mapping in problem~\eq{SR1}--\eq{SR} has a a two-parameter group of symmetries 
\begin{align*}
&S_{X_0, Y} = \{e^{sX_0} \circ e^{r Y} \mid s \in S^1, \ r \in \R\} \subset \Diff(G).
\end{align*}

Consider in the domain 
$$
G_1 = \{g \in G \mid r > 0\}, \qquad r = \sqrt{x^2+y^2},
$$
coordinates corresponding to the action the group $S_{X_0, Y}$:
\be{PQR}
P = \frac{z}{2r^2}, \quad
Q = \frac{xv+yw}{r^4}, \quad
R = \frac{-yv+xw}{r^4}, \quad
r, \quad \chi = \arctan \frac xy.
\ee
We have
\begin{align}
&X_0 P = X_0 Q = X_0 R = X_0 r = 0, \quad X_0 \chi = 1,  \label{X0P}\\
&Y P = Y Q = Y R = Y \chi = 0, \quad Y r = 1. \label{YP}
\end{align}
Differentiating formulas of coordinates \eq{PQR} with account of  \eq{X0P}, \eq{YP}, we compute the Jacobian of transformation to these coordinates:
$$
\frac{\partial(P, Q, R, r, \chi)}{\partial(x, y, z, v, w)} = \frac{1}{2 r^9}.
$$
Thus if $g_t = \Exp(\lam,t) \in G_1$, then 
\be{J1}
J = 
\frac{\partial(x, y, z, v, w)}{\partial(P, Q, R, r, \chi)} \cdot
\frac{\partial(P, Q, R, r, \chi)}{\partial(t, \f, k, \a, \b)}
= 2 r^9 \cdot
\frac{\partial(P, Q, R, r, \chi)}{\partial(t, \f, k, \a, \b)}.
\ee

In order to compute derivatives in \eq{J1} w.r.t. $\a$ and $\b$, notice that in the coordinates $(t, \f, k, \a, \b)$ on $N = C \times \R_+$ we have
\be{h0Z}
\vh_0 = \pder{}{\b}, \qquad
Z = \pder{}{t} - 2 \a \pder{}{\a},
\ee 
see the remark at the end of Sec.~4.2~\cite{max1}. Further, for any function $\map{f}{N}{\R}$ of the form $f = l \circ \Exp$ we have
\begin{align*}
\restr{\vh_0f}{\nu} &= 
\restr{\der{}{s}}{s=0} f \circ  e^{s \vh_0} (\nu) = 
\restr{\der{}{s}}{s=0} l \circ \Exp \circ  e^{s \vh_0} (\nu) \\
&= 
\restr{\der{}{s}}{s=0} l \circ e^{s X_0} \circ \Exp (\nu) = 
\restr{X_0 l}{g}, \qquad g = \Exp(\nu), \qquad \nu \in N.
\end{align*}
We used here Eq. \eq{X0h0}.

We get similarly a differentiation rule
$$
\restr{Zf}{\nu} = \restr{Yl}{g}, \qquad g = \Exp(\nu), \qquad \nu \in N.
$$
In view of \eq{X0P}, \eq{YP}, the functions $P$, $Q$, $R$ are invariants of the group of symmetries $S_{X_0, Y}$, whence
$$
\vh_0 P = \vh_0 Q = \vh_0 R = Z P = ZQ = ZR = 0.
$$
Similarly, in view of the equalities $X_0 r = 0$, $X_0 \chi = 1$, $Y r = 1$, $Y \chi = 0$, we get
$$
\vh_0 r = 0, \quad \vh_0 \chi = 1, \quad Z r = r, \quad Z \chi = 0.
$$
We obtain from representations \eq{h0Z} that
$$
\pder{f}{\a} = f_{\a} = - \frac{1}{2 \a}(Zf - t f _t), \qquad
\pder{f}{\b} = f_{\b} = \vh_0f, \qquad 
f \in C^{\infty}(N).
$$
Now we are able to transform the Jacobian:
\begin{align*}
\pder{(P, Q, R, r, \chi)}{(t, \f, k, \a, \b)} 
&= 
\begin{vmatrix}
P_t & P_{\f} & P _k & P_{\a} & P_{\b} \\
\vdots & \vdots & \vdots & \vdots & \vdots \\
\chi_t & \chi_{\f} & \chi _k & \chi_{\a} & \chi_{\b}
\end{vmatrix} =
-\frac{1}{2 \a} 
\begin{vmatrix}
P_t & P_{\f} & P _k & Z P & \vh_0 P \\
\vdots & \vdots & \vdots & \vdots & \vdots \\
\chi_t & \chi_{\f} & \chi _k & Z \chi & \vh_0 \chi
\end{vmatrix} \\
&=
-\frac{1}{2 \a} 
\begin{vmatrix}
P_t & P_{\f} & P _k & 0 & 0 \\
Q_t & Q_{\f} & Q _k & 0 & 0 \\
R_t & R_{\f} & R _k & 0 & 0 \\
* & * & *  & r  & 0 \\
* & * & * & 0& 1
\end{vmatrix} = 
-\frac{r}{2 \a} \pder{(P, Q, R)}{(t, \f, k)}.
\end{align*}
In view of equality \eq{J1}, we obtain
\be{J1r10al}
J = - \frac{r^{10}}{\a} \cdot \pder{(P, Q, R)}{(t, \f, k)}
\ee
in the case $r > 0$.

\subsection{Computation of Jacobian $J$}\label{subsec:JacC1}
Immediate computation on the basis of explicit parameterization of the exponential mapping (Subsec.~5.3~\cite{dido_exp}) gives the following result:
\be{dPQRdtpk}
\pder{(P, Q, R)}{(t, \f, k)} = \frac{16 k}{2(1-k^2)} \cdot \frac{1}{r^{10} \Delta^2} \cdot J_1,
\ee
where
\begin{align}
&\Delta = 1 - k^2 \sn^2 p \, \sn^2 \tau, \nonumber \\
&p = \sqrt{\a} \frac t2, \qquad \tau = \sqrt{\a}\left(\f + \frac t2\right), \nonumber\\
&J_1 = a_0 + \xi a_1 + \xi^2 a_2, \qquad \xi = \sn^2 \tau, \label{J2=a0x}\\
&a_0 = f_V \cdot a_{01}, \qquad a_2 = f_z \cdot a_{21} \label{a0a2},\\
&a_1 = -  a_0 - a_2/k^2,   \label{a1}
\intertext{the functions $f_z$ and $f_V$ are defined in \eq{fzC1} and \eq{fVC1},}
&a_{01} = 
(3 \cn p (-3 E_2^2(p) + 4 E_2(p) p - 8 E_2(p) k^2 p - p^2 + 
    2 k^2 (-4 + 3E_2^2(p)  \nonumber \\
	&\qquad 	+ 4E_2(p)(-1 + 2k^2)p + p^2) \sn^2(p) + 8 k^4 \sn^4(p)) \nonumber \\
	&\qquad + 
  3 \dn p \sn p (-E_2^3(p) - 2p + E_2^2(p)(2 - 4k^2)p + 8k^2p(1 + (1 - 2k^2) \sn^2(p)) \nonumber \\
	&\qquad - 
    E_2(p) (-2 + p^2 + 4 k^2(-2 + 3 \sn^2 p))))/4,
\label{b0}\\
&a_{21} = 
-(k^2( \cn p (-E_2^5(p) + E_2^4(p) (2 - 4k^2)p - E_2^2(p) (9 - 64k^2 + 64k^4)p \nonumber\\
&\qquad + p^3 - 
     E_2^3(p) (-8 + 16k^2 + p^2)) + \dn p (E_2^2(p)(-4 + 5E_2^2(p) + 48k^2) \nonumber \\
		&\qquad + 
     2E_2(p)(1 - 4E_2^2(p) + 8(-6 + E_2^2(p))k^2 + 64k^4)p \nonumber \\
		&\qquad + (2 + 3E_2^2(p))p^2) \sn p - 
   16 k^2 \cn p   (5E_2(p) - 2p + 4k^2p) \sn^2 p \nonumber \\
	&\qquad - 
   4 k^2 \dn p (-12 + 10 E_2^2(p) + 8E_2(p)(-1 + 2k^2)p + p^2) \sn^3 p \nonumber \\
	&\qquad + 
   16 k^4 \cn p (5 E_2(p) - 2p + 4k^2p) \sn^4 p - 48 k^4 \dn p  \sn^5 p )),
\label{b2}\\
&E_2(p,k) = 2 \E(p,k)- p. \nonumber
\end{align}
Now from equalities \eq{J1r10al}, \eq{dPQRdtpk} we get a factorization
\be{J1=fJ2}
J = - \frac{16k}{3(1-k^2) \a \Delta^2} \cdot J_1
\ee
under the assumption $r > 0$.

\begin{remark}
The function $t \mapsto r^2 = x^2 + y^2$ is real analytic and is not identically zero, thus its roots are isolated. The both sides of equality~\eq{J1=fJ2} are continuous in $t$, so this equality holds for all $(\lam,t) \in C_1 \times \R_+$.
\end{remark}

Summing up, we have the following statement.
\begin{lemma}\label{lem:signJ}
For all $(\lam,t) \in C_1 \times \R_+$ we have $\sgn J = \sgn J_1$.
\end{lemma}

\subsection{Estimate of Jacobian $J_1$}

\begin{lemma}\label{lem:J2<0}
Consider the function~\eq{J2=a0x}
$$
J_1 = a_0 + \xi \cdot a_1 + \xi^2 \cdot a_2, \qquad  \xi \in [0, 1],
$$
where the coefficients $a_0$, $a_1$, $a_2$ are given by equalities~\eq{a0a2}, \eq{b0}, \eq{b2}, \eq{a1}, and $k \in (0,1)$. 

Then $J_1 < 0$ for all $p \in (0, p_1(k))$, where the function $p_1(k)$ is given by Eq.~\eq{p1C1}.  
\end{lemma}
\begin{proof}
1) Let us show that $a_0(p) < 0$ for $p \in (0, p_1(k))$.

We have from \eq{a0a2} that $a_0 = f_V \cdot a_{01}$. Immediate computation shows that
\begin{align}
&\pder{}{p}\left(\frac{a_{01}}{f_z} \right) f_z^2 = \frac 34 x_2, \label{dpb0fz}\\
&x_2 = k^2 (\cn p E_4(p) p - \a_0)^2 + (1-k^2)(E_2(p) p + \b_0)^2, \label{x2=k2}\\
&\a_0 = (1 + \sn^2(p) - 2 k^2 \sn^2 p)E_2^2(p) - 4 (2 k^2-1) \cn p \sn p \dn p E_2(p) \nonumber\\
&\qquad + 4(2k^2-1)\sn^2 p \dn^2 p, \nonumber\\
&\b_0 = (2 k^2\sn^2 p-1) E_2^2(p) + 8 k^2 \cn p \sn p \dn p E_2(p) - 8 k^2 \sn^2 p \dn^2 p, \nonumber\\
&E_4(p) =  \cn p E_2(p) - 2 \sn p \dn p, \qquad E_2(p) = 2 \E(p) - p.  \nonumber
\end{align}
Equality \eq{x2=k2} means that $x_2 \geq 0$. Moreover, the identity $\ds x_2 {\equiv}_{\substack{p}} 0$  is possible only when $\cn p E_4(p) p - \a_0 {\equiv}_{\substack{p}} 0$, $E_2(p) + \b_0 {\equiv}_{\substack{p}} 0$, which is impossible in view of the expansions
\begin{align*}
&\cn p E_4(p) p - \a_0  = - 2 p^2 + o(p^2), \\
&E_2(p) + \b_0 = -4/45 k^2 p^6 + o(p^6), \qquad p \to 0.
\end{align*}
Thus $x_2 \geq 0$, moreover, $x_2(p) = 0$ only at isolated points $p$. 

Then equality \eq{dpb0fz} means that the function $p \mapsto \ds\frac{a_{01}(p)}{f_z(p)}$ strictly increases at each interval where $f_z(p) \neq 0$. Determine the sign of $f_z(p)$ at the first such interval $p \in (0, p_z^1)$:
$$
f_z(p) = \frac{p^3}{3} + o(p^3), \quad p \to 0 \qquad \then \qquad f_z(p) > 0 \text{ for } p \in (0, p_z^1),
$$ 
see the plots of $f_z(p)$ at Figs. 3--4~\cite{max3}.

Determine similarly the sign of $a_{01}(p)$ for small $p > 0$:
$$
a_{01}(p) = \frac{4}{1575} k^2 (1-k^2)p^{10} = o(p^{10}), \quad p \to 0,
$$
thus $a_{01}(p) > 0$ for $p \to + 0$. Consequently, $\ds\frac{a_{01}}{f_z} > 0$ for $p \to + 0$. But $\ds\frac{a_{01}}{f_z}$ strictly increases at the interval $(0, p_z^1)$, thus $\ds\frac{a_{01}}{f_z} > 0$ when $p \in (0, p_z^1)$. Since $f_z(p)> 0$, then
\be{b0>0}
a_{01}(p) > 0 \text{ for } p \in (0, p_z^1).
\ee

Finally, let us determine the sign of   $f_V(p)$ at the first interval of sign-definiteness $p \in (0, p_V^1)$:
\be{fV<0}
f_V(p) = - \frac{4}{45} p^6 + o(p^6) < 0 \text{ as } p \to + 0,
\ee
see the plots of $f_V(p)$ at Figs. 8, 9~\cite{max3}.

Now conditions \eq{a0a2}, \eq{b0>0}, \eq{fV<0} imply the required inequality:
\be{a0<0}
a_0(p) < 0 \text{ for } p \in (0, p_1(k)).
\ee

\medskip

2) Let us show that $a_2(p) > 0$ for  $p \in (0, p_1(k))$.
Condition~\eq{a0a2} gives $a_2 = f_z \cdot a_{21}$. Immediate differentiation yields the equalities
\begin{align}
&\pder{}{p} \left(  \frac{a_{21}(p)}{f_V(p)}\right) f_V^2(p) = - \frac 43 k^2 x_1, \nonumber\\
&x_1 = k^2(\cn^2 p \, p^2 + \b_1 p + \gamma_1)^2 + (1-k^2)(p^2 + \delta_1 p + \eps_1)^2, \label{x1=k2}\\
&\b_1 = 
- \cn^2 p  E_2^3(p) + 6 \cn p \sn p \dn p E_2^2(p) - (8 - 10 \cn^2(p) \nonumber\\
&\qquad - 4 k^2 \sn^2(p) (2 - 3 \cn^2(p))) E_2(p) + 
 4  \cn p \sn p \dn p (2 k^2 \sn^2 p - 1), \nonumber
\\
&\gamma_1 = 
8 \cn p \dn p E_2^3(p) (-1 + 2 k^2) \sn p + E_2^4(p) (3 - \sn^2 p + 2 k^2 (-2 + \sn^2 p)) \nonumber \\
&\qquad - 
 4 \dn^2 p \sn^2 p (3 + 8 k^4 (2 + \sn^2 p) - 4 k^2 (5 + \sn^2 p)) \nonumber \\
&\qquad - 
 4 \cn p \dn p E_2(p) \sn p (-7 + 8 k^2 (5 + \sn^2 p - 2 k^2 (2 + \sn^2 p))) \nonumber \\
&\qquad + 
 E_2^2(p) (-15 + 23 \sn^2 p + 
    8 k^2 (10 - 10 \sn^2 p - 3 \sn^4 p \nonumber \\
&\qquad+ k^2 (-8 + 4 \sn^2 p  + 6 \sn^4 p))), \nonumber
\\
&\delta_1 = 
-E_2^3(p) - (2 - 4 k^2 \sn^2 p) E_2(p), \nonumber
\\
&\eps_1 = 
16 \cn p  \dn p  E_2^3(p) k^2 \sn p + E_2^4(p) (1 + 2 k^2 (-2 + \sn^2 p)) \nonumber \\
&\qquad+ 
32 \cn p \dn p  E_2(p) k^2 \sn p  (-3 + 2 k^2 (2 + \sn^2 p)) \nonumber \\
&\qquad - 
 16 \dn^2 p k^2 \sn^2 p  (-3 + 2 k^2 (2 + \sn^2 p)) \nonumber \\
&\qquad+ 
 E_2^2(p) (1 + 16 k^2 (3 - 4 \sn^2 p + k^2 (-4 + 2 \sn^2 p + 3 \sn^4 p))). \nonumber 
\end{align}
Thus $x_1 \geq 0$, moreover, $x_1$ vanishes at isolated points in view of the following expansions as $p \to 0$:
\begin{align*}
&p^2 \cn^2 p   + \b_1 p + \gamma_1 = \frac{4}{1575}(1-k^2) p^{10} + o(p^{10}) \not\equiv 0, \\
&p^2 + \delta_1 p + \eps_1 = \frac{4}{1575} k^2 p^{10} + o(p^{10}) \not\equiv 0.
\end{align*}
Consequently, the function $\ds \frac{a_{21}(p)}{f_V(p)}$ strictly decreases at each interval where $f_V(p) \neq 0$. Taking into account the signs:
\begin{align*}
&f_V(p) < 0 \text{ for } p \in (0, p_V^1), \\
&a_{21}(p) = \frac{16}{1488375} k^4 (1-k^2)p^{15} + o(p^{15}) > 0 \text{ as } p \to + 0,
\end{align*}
we conclude that $\ds\frac{a_{21}(p)}{f_V(p)}< 0$ as $p \to + 0$, thus $\ds\frac{a_{21}(p)}{f_V(p)}< 0$  for $p \in (0, p_V^1)$. Whence ${a_{21}(p)}> 0$  for $p \in (0, p_V^1)$, so
\be{a2>0}
a_2(p) > 0 \text{ for } p \in (0, p_1(k)).
\ee

\medskip

3) Let us prove that $a_0 + a_1 + a_2 < 0$ for $p \in (0, p_1(k))$. We obtain from~\eq{a1} the equalities
\begin{align*}
&a_0 + a_1 + a_2 = (1-k^2) f_z \cdot a_{021}, \\
&a_{021} = - a_{21}/k^2.
\end{align*}
Immediate differentiation gives
$$
\pder{}{p}\left(\frac{a_{021}}{f_V}\right) f_V^2  = \frac 43 x_1,
$$
where the function $x_1$ is defined by Eq. \eq{x1=k2}. It was shown in item 2) of this proof that $x_1(p) $ is nonnegative and vanishes at isolated points. Thus the function $\ds\frac{a_{021}(p)}{f_V(p)}$ strictly increases at each interval where $f_V(p) \neq 0$. Taking into account the sign
$$
a_{021} = - \frac{16}{1488375} k^2(1-k^2)p^{15} + o(p^{15}) < 0 \text{ as } p \to 0,
$$
we conclude that $\ds\frac{a_{021}(p)}{f_V(p)}>0 $ for $p \in (0, p_1(k))$, thus ${a_{021}(p)}<0 $ for $p \in (0, p_1(k))$ and
\be{a012<0}
a_0 + a_1 + a_2 < 0 \text{ for } p \in (0, p_1(k)).
\ee

\medskip

4) Let us make use of the bounds on the coefficients $a_0$, $a_1$, $a_2$ proved in items 1)--3) and show that the quadratic trinomial $J_1(\xi) = a_0 + a_1 \xi + a_2 \xi^2$ is negative at the segment $\xi \in [0, 1]$.

Choose any $k \in (0,1)$ and $p \in (0, p_1(k))$.
At the endpoints of the segment $\xi \in [0, 1]$ we have:
\begin{align*}
&J_1(0) = a_0 < 0  \text{ by } (\ref{a0<0}),\\
&J_1(1) = a_0 + a_1 + a_2 < 0 \text{ by } (\ref{a012<0}).
\end{align*}
But the equality $a_2 > 0$ (see \eq{a2>0}) means that the function $\xi \mapsto J_1(\xi)$ is convex. Thus
$$
J_1(\xi) \leq \xi a_0 + (1-\xi) (a_0+a_1+a_2) < 0 \text{ for } \xi \in [0, 1].
$$
Lemma \ref{lem:J2<0} is proved.
\end{proof}

\subsection{Proof of the bound of conjugate time for $\lam \in C_1$}
Theorem \ref{th:tconjC1} follows immediately from Lemmas \ref{lem:signJ} and \ref{lem:J2<0}.

\section{Conjugate time for $\lam \in C_2$}\label{sec:C2}
In this section we prove Theorem \ref{th:tconj} for $\lam \in C_2$:

\begin{theorem}\label{th:tconjC2}
If $\lam \in C_2$, then $\tconj(\lam) \geq \tmax(\lam)$.
\end{theorem}

\subsection{Computation of Jacobian $J$}\label{subsec:JacC2}
We use the elliptic coordinates $(\p, k, \a, \b)$ in the domain $C_2$, see Subsec.~\ref{subsec:previous}. For a fixed $\lam = (\p, k, \a, \b) \in C_2$, conjugate times are roots $t>0$ of the Jacobian
$$
J = \pder{(x, y, z, v, w)}{(t, \p, k, \a, \b)}.
$$
We compute this Jacobian in the domain $G_1 = \{ g \in G \mid r^2 = x^2 + y^2 > 0\}$ in the same way as in Subsec~\ref{subsec:JacC1}:
\begin{align}
&J = - \frac{r^{10}}{\a} \pder{(P, Q, R)}{(t, \p, k)} = - \frac{64}{3k^4(1-k^2)\Delta^2 \a} J_1, \label{JJ1}\\
&J_1 = a_0 + a_1 \xi + a_2 \xi_2, \qquad \xi = \sin^2 u_2, \label{J1C2}\\ 
&p = \frac{t}{2k}, \qquad \tau = \p + \frac{t}{2k}, \nonumber \\
&u_1 = \am(p,k), \qquad u_2 = \am(\tau,k), \nonumber \\
&a_0 = \frac{1}{16} f_V \cdot a_{01}, \nonumber \\
&a_0+a_1+a_2 = (1-k^2)a_0 = \frac{1-k^2}{16} f_V \cdot a_{01}, \nonumber \\
&a_2 = f_z \cdot a_{21} \label{a2}, \\
&f_z = 
\frac{2}{k}
 \left(\sqrt{1-k^2 \sin^2 u_1} ((2 - k^2)F(u_1) - 2 E(u_1)) + k^2 \cos u_1 \sin u_1 \right) , \nonumber \\
&f_V =
\dfrac 43 \left\{ 3\sqrt{1-k^2 \sin^2 u_1} \,(2E(u_1) - (2 - k^2)F(u_1))^2   \right. \nonumber\\
&\qquad+ \cos u_1
[8E^3(u_1) - 4E(u_1)(4 + k^2)  -
 12 E^2(u_1)(2-k^2) F(u_1)  \nonumber\\
&\qquad
 + 6 E(u_1)(2-k^2)^2 F^2(u_1) + 
 F(u_1)(16 - 4k^2-3k^4  \nonumber \\
 &\qquad -(2-k^2)^3 F^2(u_1))] \sin u_1 
 -
2 \sqrt{1-k^2 \sin^2 u_1}  \,(-4k^2+3(2 E(u_1)
\nonumber\\
& \qquad
-(2-k^2)F(u_1))^2) \sin^2 u_1 
+ 12 k^2 \cos u_1 (2 E(u_1) \nonumber\\
&\qquad
 \left. - (2 - k^2) F(u_1)) \sin^3 u_1 
- 8 k^2 \sin^4 u_1 \sqrt{1-k^2 \sin^2 u_1}  \right\},  \nonumber \\
&k^4a_0 + k^2 a_1 + a_2 = (1-k^2)a_2 = (1-k^2) f_z \cdot a_{21},  \nonumber \\
&a_1 = - k^2 a_0 - a_2, \label{a102}
\end{align}
where the coefficients $a_{01}$ and $a_{21}$ are defined explicitly in~\eq{a01C2} and \eq{a21C2}.

\begin{remark}
Equality \eq{JJ1} is valid for all $(\lam, t) \in C_2 \times \R_+$ by the same argument as in remark at the end of Subsec.~$\ref{subsec:JacC1}$.
\end{remark}

\subsection{Conjugate time as $k \to 0$}
Compute the asymptotics of the Jacobian $J_1$ \eq{J1C2} as $k \to 0$:
\begin{align*}
&f_z = f_z^0 k^3 + o(k^3), \\
&f_z^0 = \frac{1}{16}(4p-\sin 4p), \\
&f_V = \frac{k^8}{512} f_V^0 + o(k^8), \\
&f_V^0 = (32 u_1^2 -1) \cos 2 u_1 + \cos 6 u_1 - 8 u_1 \sin 2 u_1, 
\end{align*}
\begin{align*}
&a_{01} = \frac{3}{2048} k^8 a_{01}^0 + o(k^8), \\
&a_{01}^0 = 64 u_1^3 \sin 2 u_1 + 48 u_1^2 \cos 2 u_1 - 44 u_1 \sin 2 u_1 - 4 u_1 \cos u_1 \sin 2 u_1 \\
&\qquad + 3 \cos 2 u_1 - 3 \cos 6 u_1, \\
&a_{21} = \frac{1}{4194304} k^{17} a_{21}^0 + o(k^{17}),\\
&a_{21}^0 = 45 u_1 + 608 u_1^3 - 512 u_1^5 + 16 u_1(28u_1^2 - 3) \cos 4 u_1 + 3 u_1 \cos 8 u_1 \\
&\qquad + 12 \sin 4 u_1 - 432 u_1^2 \sin 4 u_1 + 256 u_1^4 \sin 4 u_1 - 6 \sin 8 u_1,
\end{align*}
thus
\begin{align}
&a_0 = k^{16} a_0^0 + o(k^{16}), \nonumber\\
&a_0^0 = \frac{3}{2^{24}} f_V^0 \cdot a_{01}^0, \label{a00}\\
&a_2 = k^{20} a_2^0 + o(k^{20}), \label{a2k->0}\\
&a_2^0 = f_z^0 \cdot a_{21}^0, \label{a20}\\
&a_1 = - k^2 a_0 - a_2 = - k^{18} a_0^0 + o(k^{18}). \nonumber
\end{align}
Then
\begin{align}
J_1 &=
k^{16} a_0^0 + o(k^{16}) + \xi (-k^{18} a_0^0 + o(k^{18})) + \xi^2(k^{20} a_2^0 + o(k^{20})) \nonumber\\
& = k^{16} a_0^0 + o(k^{16}) = \frac 18 k^{16} f_V^0 \cdot a_{01}^0 + o(k^{16}). \label{J1k->0}
\end{align}
Thus, in order to bound the first positive root of $J_1$ as $k \to 0$, we have to bound the first positive roots of $f_V^0$ and $a_{01}^0$.

Let us denote by $F[a]$ the first positive root of a function $a(t)$:
$$
F[a] := \inf \{ t > 0 \mid a(t) = 0 \}.
$$

\begin{lemma}\label{lem:fV20}
We have $F[f_V^0] \in (\frac 58 \pi, \frac 34 \pi)$, moreover, the function $f_V^0$ has a simple root at the point $F[f_V^0]$. If $u_1 \in (0, F[f_V^0])$, then $f_V^0(u_1) < 0$.
\end{lemma}
\begin{proof}
Introduce the function
$$
g(u_1) = \frac{2 f_V^0(u_1)}{\sin 2 u_1}, \qquad u_1 \neq \frac{\pi n}{2}.
$$
We have 
$$
g'(u_1) = - \frac{8 (\sin 4 u_1 - 4 u_1)^2}{\sin^2 2 u_1}, \qquad  u_1 \neq \frac{\pi n}{2},
$$
thus $g(u_1)$ decreases at intervals $u_1 \in \left(\frac{\pi n}{2}, \frac{\pi (n+1)}{2}\right)$. Since $\lim_{u_1 \to 0} g(u_1) = 0$, then $g(u_1) < 0$ and $f_V^0(u_1) < 0$ when $u_1 \in (0, \frac{\pi}{2})$. Evaluate $f_V^0$ at several points:
\begin{align*}
&f_V^0 \left(\frac{\pi}{2}\right) = - 8 \pi^2 < 0, \\
&f_V^0 \left(\frac{3 \pi}{4}\right) = 6 \pi > 0, \\
&f_V^0 \left(\frac{5 \pi}{8}\right) = \frac{4 + 10 \pi - 25 \pi^2}{2 \sqrt 2} < 0.
\end{align*}
By monotonicity of $g$, it follows that $F[f_V^0] \in (\frac 58 \pi, \frac 34 \pi)$. It is obvious from the above that  $f_V^0(u_1) < 0$ for $u_1 \in (0, F[f_V^0])$. 

Let us prove that the first root is simple:
$$
(f_V^0(u_1))' = (1/2 \sin 2 u_1 g(u_1))' = 2 \cot 2 u_1 f_V^0 (u_1) + 1/2 \sin 2 u_1 g'(u_1).
$$
If $u_1 = F[f_V^0]$, then $f_V^0(u_1) = 0$, $g'(u_1) < 0$, $\sin 2 u_1 < 0$, thus $(f_V^0)'(u_1) > 0$. 
\end{proof}

\begin{lemma}\label{lem:a010}
We have $F[a_{01}^0] \in (\frac 34 \pi, \pi)$. If $u_1 \in (0, F[a_{01}^0])$, then $a_{01}^0(u_1) < 0$.
\end{lemma}
\begin{proof}
We have $a_{01}^0(\pi) = 48 \pi^2 > 0$, thus it remains to prove that $a_{01}^0(u_1) < 0$ for $u_1 \in (0, \frac 34 \pi]$. 

We prove this bound by the method \, {\em ''divide et impera''}.  
Let us apply this method to the function
\begin{align*}
&x_0(u_1) = a_{01}^0(u_1).\\
\intertext{We get:}
&x_0 = 64 u_1^3 \tx_0 + o(u_1^3), \qquad u_1 \to \infty, \\
&\tx_0 = \sin 2 u_1, \\
&x_1 = \left(\frac{x_0}{\tx_0}\right)' \tx_0^2 = 
-96 u_1^2 \tx_1 + o(u_1^2), \qquad u_1 \to \infty, \\
&\tx_1 = \cos 4 u_1, \\
&x_2 = \left(\frac{x_1}{\tx_1}\right)' \tx_1^2 = 
32 u_1 \tx_2 + o(u_1), \qquad u_1 \to \infty, \\
&\tx_2 = (15+14 \cos 4 u_1 + \cos 8 u_1) \sin^2 2 u_1, \\
&x_3 = \left(\frac{x_2}{\tx_2}\right)' \tx_2^2 = 
-1024  \cos 4 u_1 \sin^{10} 2 u_1.
 \end{align*}
If $u_1 \in ( 0, \pi/8)$, then $x_3 < 0$, thus $\ds\frac{x_2}{\tx_2}$ decreases. We have $\ds \lim_{u_1 \to 0} \frac{x_2}{\tx_2} = 0$. Thus $\ds\frac{x_2}{\tx_2} < 0$ for $u_1 \in ( 0, \pi/8)$.  Since $\tx_2 > 0$ for $u_1 \in ( 0, \pi/8)$, then $x_2 < 0$, thus $\ds\frac{x_1}{\tx_1}$ decreases for $u_1 \in ( 0, \pi/8)$.
We have $\ds \lim_{u_1 \to 0} \frac{x_1}{\tx_1} = 0$. Then for $u_1 \in ( 0, \pi/8)$ we get  $\ds\frac{x_1}{\tx_1} < 0$, thus $x_1 < 0$, so $\ds\frac{x_0}{\tx_0}$ decreases. Finally,  $\ds\lim_{u_1 \to 0} \frac{x_0}{\tx_0} = 0$.  Thus for $u_1 \in ( 0, \pi/8)$ we have $\ds\frac{x_0}{\tx_0} < 0$ and $x_0 < 0$.

We prove similarly, going by steps of length $\pi/8$, that $a_{01}^0 < 0$ for $u_1 \in \left[ \frac{\pi}{8}, \frac{3\pi}{4}\right]$. 

Thus $F[a_{01}^0] \in \left( \frac{3\pi}{4}, \pi \right)$, and $a_{01}^0 < 0$ for $u_1 \in (0, F[a_{01}^0])$.
\end{proof}

\begin{lemma}\label{lem:a210}
If $u_1 > 0$, then $a_{21}^0 (u_1) < 0$. 
\end{lemma}
\begin{proof}
We apply the method \, {\em ``divide et impera''}. We have:
\begin{align*}
&x_0 = a_{21}^0 = - 512 u_1^5 + o(u_1^5), \qquad u_1 \to \infty,\\
&x_1 = x_0'  = 512 u_1^4 \tx_1+ o(u_1^4), \qquad u_1 \to \infty,\\
&\tx_1 = 2 \cos 4 u_1 - 5 < 0, \\
&x_2 = \left(\frac{x_1}{\tx_1}\right)' \tx_1^2   = 1024 u_1^3 \tx_2+ o(u_1^3), \qquad u_1 \to \infty,\\
&\tx_2 = 48 - 25 \cos 4 u_1  + 4 \cos 8 u_1 > 0, \\
&x_3 = \left(\frac{x_2}{\tx_2}\right)' \tx_2^2   = 1536 u_1^2 \tx_3+ o(u_1^2), \qquad u_1 \to \infty,\\
&\tx_3 = 
(2 \cos 4 u_1 - 5)(-781 + 96 \cos 4 u_1 - 141 \cos 8 u_1 + 16 \cos 12 u_1)
 > 0, \\
&x_4 = \left(\frac{x_3}{\tx_3}\right)' \tx_3^2   = 6144 u_1 \tx_4+ o(u_1), \qquad u_1 \to \infty,\\
&\tx_4 = 
(2 \cos 4 u_1 - 5)(48 - 25 \cos 4 u_1 + 4 \cos 8 u_1) \\
&\qquad (-25136 - 27745 \cos 4 u_1 - 3880 \cos 8 u_1 + 53 \cos 12 u_1 + 8 \cos 16 u_1) \\
& \qquad \sin^2 2 u_1
 > 0,  \qquad u_1 \neq \frac{\pi n}{2},\\
&x_5 = \left(\frac{x_4}{\tx_4}\right)' \tx_4^2   = 
786432 (5-2\cos 4 u_1)^4 (48 - 25 \cos 4 u_1 + 4 \cos 8 u_1)^2  \\
& \qquad
(98009 - 18812 \cos 4 u_1 + 18073 \cos 8 u_1 - 2564 \cos 12 u_1 + 64 \cos 16 u_1)  \\
& \qquad\sin^{12} 2 u_1
> 0, 
\qquad u_1 \neq \frac{\pi n}{2}.
\end{align*}
The function $\ds\frac{x_4}{\tx_4}$ increases on $\R$, and since $\ds\lim_{u_1 \to 0} \frac{x_4}{\tx_4} = 0$, then $\ds\frac{x_4}{\tx_4} > 0$ for $u_1 > 0$. Let $u_1 > 0$, then $x_4 > 0$, thus $\ds\frac{x_3}{\tx_3}$ increases. Since   $\ds\lim_{u_1 \to 0} \frac{x_3}{\tx_3} = 0$, then $\ds \frac{x_3}{\tx_3} > 0$, $x_3 > 0$, and $\ds \frac{x_2}{\tx_2}$ increases for $u_1 > 0$. By virtue of $\ds\lim_{u_1 \to 0} \frac{x_2}{\tx_2} = 0$, we have $\ds\frac{x_2}{\tx_2} > 0$, $x_2 > 0$, and $\ds\frac{x_1}{\tx_1}$ increases as $u_1 > 0$. Since $\ds\lim_{u_1 \to 0} \frac{x_1}{\tx_1} = 0$, then $\ds\frac{x_1}{\tx_1} > 0$, $x_1 < 0$, and $x_0$ decreases as $u_1 > 0$. Finally, we have $x_0(0) = 0$, thus $x_0(u_1) = a_{21}^0(u_1) < 0$ for $u_1 > 0$.
\end{proof}

Now we are able to bound the Jacobian $J_1$ for $k \to 0$ as follows.
\begin{lemma}\label{lem:J1k->0}
For any $\eps>0$ there exists $\bk = \bk(\eps)\in (0,1)$ such that for any 
$k \in (0, \bk]$, any $u_1 \in (0, u_V^1(k)-\eps]$ and any 
$\xi \in [0, 1]$ we have $J_1(u_1, \xi, k) > 0$.
\end{lemma}
\begin{proof}
By virtue of \eq{J1k->0},
\begin{align*}
&J_1 = k^{16} a_0^0 + o(k^{16}), \qquad k \to 0, \\
&a_0^0 = \frac{3}{2^{24}} f_V^0 \cdot a_{01}^0,
\end{align*}
moreover,
\be{J1u1xk16}
|J_1(u_1, \xi, k) - k^{16}a_0^0(u_1)| \leq f(k) = o(k^{16}), \qquad k \to 0,
\ee
where the function $f(k)$ does not depend on $\xi$ since $\xi \in [0, 1]$. 

By Lemmas \ref{lem:fV20}, \ref{lem:a010}, if $u\in(0, u_V^1(0))$, then $a_0^0(u_1) > 0$.

Define a function:
$$
\tJ(u_1, \xi, k) = 
\begin{cases}
J_1(u_1, \xi, k)/k^{16}, & k>  0, \\
a_0^0(u_1), & k = 0.
\end{cases}
$$
We have $\lim_{k \to 0}J_1/k^{16} = a_0^0$, thus the function $\tJ$ is continuous on the set $\R_{u_1} \times [0, 1]_{\xi} \times [0, 1)_k$.  Then there exists an open set $O_1 \subset \R_{u_1} \times [0, 1)_k$ such that:
\begin{align*}
&O_1 \supset I := \{(u_1, k) \in \R \times [0, 1) \mid u_1 \in (0, u_V^1(0), \ k = 0\}, \\
&\restr{\tJ}{O_1} > 0.
\end{align*}
By virtue of inequality~\eq{J1u1xk16}, the neighborhood $O_1$ does not depend on $\xi$. Then the set $O_2 = O_1 \setminus I \subset \R_{u_1} \times (0, 1)_k$ is open and $\restr{J_1}{O_2} > 0$.

Now we study the sign of $J_1$ in the neighborhood of $(u_1, k) = (0, 0)$. To this end, compute the asymptotics as $u_1^2 + k^2 \to 0$:
\begin{align*}
&a_0 = \frac{4}{70875} k^{16} u_1^{16} + o(u_1^2+k^2)^{16}, \\ 
&a_2 = -\frac{4}{4465125} k^{20} u_1^{20} + o(u_1^2+k^2)^{20}, \\
&a_1 = - k^2 a_0 - a_2 
= - \frac{4}{70875} k^{18} u_1^{16} + o(u_1^2+k^2)^{17}, \\
\intertext{whence}
&J_1 =   \frac{4}{70875} k^{16} u_1^{16} + o(u_1^2+k^2)^{16} > 0 \text{ as } (u_1, k) \to (0, 0).
\end{align*}
Thus there exists an open set $O_3 \subset (0, + \infty)_{u_1} \times (0,1)_k$ such that the closure of $O_3$ contains a neighborhood of $(u_1, k) = (0, 0)$ in $[0, + \infty)_{u_1} \times [0, 1)_k$ and $\restr{J_1}{O_3} > 0$. The set $O_3$ does not depend on $\xi$ by the same argument as $O_1$, $O_2$. 

Take any $\eps > 0$. The set $\cl(O_2) \cup \cl(O_3)$ contains a neighborhood of the segment $[0, u_V^1(0) - \eps]_{u_1} \times \{k = 0\}$. Thus there exists $\bk = \bk(\eps) \in (0, 1)$ such that the domain $\{(u_1, k ) \in \R \times (0, 1) \mid k \in (0, \bk], \ u_1 \in (0, u_V^1(k) - \eps]\}$ is contained in $O_2 \cup O_3$, thus the Jacobian $J_1$ is positive on it:
$$
\forall k \in (0, \bk] \quad \forall u_1 \in (0, u_V^1(k) - \eps] \quad \forall \xi \in [0, 1] \qquad J_1(u_1, \xi, k) > 0.
$$
\end{proof}

The bound of the Jacobian $J_1$ of Lemma~\ref{lem:J1k->0} can be reformulated in terms of conjugate points along the corresponding geodesics.

\begin{cor}\label{cor:indk->0}
Take any $\eps > 0$ and choose $\bk = \bk(\eps) \in (0, 1)$ according to Lemma~$\ref{lem:J1k->0}$. Take any $\lam = (\p, k, \a, \b) \in C_2$ with $k \in (0, \bk]$. Define functions:
\begin{align*}
&u_1(k, \eps) = u_V^1(k) - \eps, \\
&p(k, \eps) = F(u_1(k,\eps), k), \\
&t(k, \eps) = 2 k p(k, \eps).
\end{align*}
Then the geodesic
$$
\gamma_{\lam}^{\eps} = \{ \Exp(\lam,s) \mid s \in (0, t(k, \eps))\}
$$
does not contain conjugate points.
\end{cor}

\subsection{Conjugate time for arbitrary $k \in (0, 1)$}
Now we can prove a bound for the Jacobian $J_1$ via homotopy from an arbitrary $k \in (0,1)$ to a small $k$ close to 0.

\begin{lemma}\label{lem:J1>0x(0,1)}
For any $k \in(0,1)$, any $u_1 \in(0, u_V^1(k))$  and any $\xi \in (0,1)$ we have $J_1(u_1, \xi, k) > 0$.
\end{lemma}
\begin{proof}
Take any $\lam_1 = (\p_1, k_1, \a_1, \b_1) \in C_2$ such that $\xi_1 \in (0, 1)$, where
\begin{align*}
&\xi_1 = \sn^2(\tau_1, k_1), \\
&\tau_1 = \p_1 + \frac{t_1}{2 k_1}, \\
&t_1 = 2 k_1 F(u_V^1(k_1), k_1).
\end{align*}
For any $k \in (0,1)$, $u_1 = u_V^1(k)$, $\xi = \xi_1$ we have
\be{J1a2xi}
J_1 = - a_2 \xi (1-\xi).
\ee
By virtue of \eq{a2}, we have $a_2(u_1, k) < 0$, thus $J_1(u_V^1(k), \xi, k) > 0$.

Recall equalities \eq{a2k->0}, \eq{a20} and define a function
$$
\ta_2(u_1, k) = 
\begin{cases}
\ds\frac{a_2(u_1, k)}{k^{20}} = f_z^0(u_1) a_{21}^0(u_1) + o(1), & k \in (0, 1), \\
 f_z^0(u_1) a_{21}^0(u_1), & k =0.
\end{cases}
$$
This function is continuous on the set $\R_{u_1} \times [0, 1)_k$. Moreover,
$$
\ta_2(u_V^1(k), k) < 0 \text{ for all } k \in [0, 1).
$$
Thus there exists $\eps> 0$ (fix it) such that for all $k \in (0, 1)$ and all $u_1 \in [u_V^1(k) - \eps, u_V^1(k)]$ we have $a_2(u_1, k) < 0$, whence $J_1(u_1, \xi_1, k) > 0$ in view of~\eq{J1a2xi}.

For the fixed $\eps> 0$, choose $\bk = \bk(\eps) \in (0, 1)$ by Lemma~\ref{lem:J1k->0}. Now we construct a homotopy of geodesics from $k_1$ to $\bk$.

A parameter of the homotopy is $\mu \in[\bk, k_1]$, we assume that $\bk < k_1$, otherwise there is nothing to prove. Define a continuous family of initial points of extremals:
\begin{align*}
&\lam^{\mu} = (\p^\mu, k = \mu, \a_1, \b_1) \in C_2, \qquad \mu \in[\bk, k_1], \\
&\xi_1 = \sn^2 \left(\p^\mu + \frac{t^\mu}{2 \mu}, \mu\right), \\
&t^\mu = 2 \mu F(u_V^1(\mu), \mu).
\end{align*}
Further, consider a continuous one-parameter family of geodesics:
\begin{align*}
&\gamma^\mu = \{\Exp(\lam^\mu, s) \mid s \in(0, t(\mu, \eps)]\}, \\
&t(\mu, \eps) = 2 \mu F(u_V^1(\mu) - \eps, \mu).
\end{align*}
By Cor.~\ref{cor:indk->0},  the geodesic $ \gamma^{\bk} $ does not contain conjugate points. 

Further, for any $\mu \in [\bk, k_1]$ and $u_1 = u_V^1(\mu)- \eps$ we have $J_1(u_1, \xi_1, \mu) > 0$ as was proved above. That is, the terminal time $s = t(\mu, \eps)$ is not conjugate for the geodesic $\gamma^\mu$.

By Th.~\ref{th:conj_hom}, 
 the curve $\gamma^{k_1}$ is free of conjugate points. 

Since $J_1(u_1, \xi_1, k) > 0$ for $u_1\in (0, u_V^1(k) - \eps]$ and small $k$, then $J_1(u_1, \xi_1, k_1) > 0 $ for $u_1 \in (0, u_V^1(k_1))$. 
\end{proof}

In the following two lemmas we clarify the sign of the Jacobian $J_1$ for the remaining special values of the parameters $\xi$, $u_1$.

\begin{lemma}
\label{lem:J1xi01}
Let $\xi \in \{0, 1\}$, $k \in (0,1)$, and $u_1 \in (0, u_V^1(k))$. Then $J_1(u_1, \xi, k) > 0$.
\end{lemma}
\begin{proof}
Take any $\lam_1 = (\p_1, k_1, \a_1, \b_1) \in C_2$ such that
\begin{align*}
&\xi_1 = \sn^2 \left(\p_1 + \frac{t_1}{2 k_1}, k_1\right) = 0, \\
&t_1 = 2 k_1 F(u_V^1(k_1), k_1),
\end{align*}
the case $\xi_1 = 1$ is considered similarly.
Take any sufficiently small $\eps> 0$. Then $\bt = t_1 - \eps$ is close to $t_1 > \bt$. Thus $\bar{\xi} = \sn^2(\p_1 + \frac{\bt}{2k_1}, k_1) \neq \xi_1 = 0$, $\bar{\xi}$ is close to $\xi_1$, thus $\bar{\xi} \in (0, 1)$. 

By Lemma \ref{lem:J1>0x(0,1)}, the geodesic $\{\Exp(\lam_1, s) \mid s \in (0, \bt]\}$ is free of conjugate points. Since $\bt = t_1 - \eps$ is arbitrarily close to $t_1$, the geodesic $\{\Exp(\lam_1, s) \mid s \in (0, t_1)\}$ is also free of conjugate points.
\end{proof}

\begin{lemma}\label{lem:J1xi=01}
Let $k \in (0,1)$, $u_1 = u_V^1(k)$, $\xi \in \{0, 1\}$. Then $J_1(u_1, \xi, k) = 0$.
\end{lemma}
\begin{proof}
By virtue of \eq{a102},
$$
J_1 = a_0(1-k^2 \xi) - a_2\xi(1-\xi),
$$
which vanishes for $u_1 = u_V^1(k)$ and $\xi \in \{0, 1\}$ in view of~\eq{a00}.
\end{proof}

\subsection{Final bound for the  conjugate time in $C_2$}
Now we summarize our study of the first conjugate time for $\lam \in C_2$.

\begin{theorem}\label{th:tconjC2x}
Let $\lam = (\p, k, \a, \b) \in C_2$ and
\begin{align*}
&t_1 = 2 k F(u_V^1(k), k) = \tmax(\lam), \\
&\xi = \sn^2\left(\p + \frac{t_1}{2k}, k\right).
\end{align*}

If $\xi \in \{0, 1\}$, then $\tconj(\lam) = t_1$.

If $\xi \in (0, 1)$, then $\tconj(\lam) \geq t_1$.
\end{theorem}
\begin{proof}
Apply Lemmas \ref{lem:J1>0x(0,1)}--\ref{lem:J1xi=01}.
\end{proof}

Now Th. \ref{th:tconjC2} follows immediately from Th.~\ref{th:tconjC2x}.

\section{Conjugate time for $\lam \in \cup_{i=3}^7 C_i$}\label{sec:C47}

\begin{theorem}\label{th:tconjC47}
If $\lam \in C_4 \cup C_5 \cup C_7$, then $\tcut(\lam)  = \tconj(\lam) = \tmax(\lam)  = + \infty$.
\end{theorem}
\begin{proof}
If $\lam \in C_4 \cup C_5 \cup C_7$, then $(x_t, y_t)$ is a straight line, thus $\Exp(\lam,t)$ is optimal for $t \in [0, + \infty)$.
\end{proof}

\begin{theorem}\label{th:tconjC3}
If $\lam \in C_3$, then $\tcut(\lam)  = \tconj(\lam) = \tmax(\lam) = + \infty $.
\end{theorem}
\begin{proof}
We project the SR problem on the Cartan group to the SR problem on the Engel group studied in~\cite{engel, engel_conj, engel_cut, engel_synth}. For $\lam \in C_3$, the projection to the Engel group is optimal, thus the geodesic in the Cartan group is optimal as well.
Let us prove this in detail.

The quotient $H = G /e^{\R X_5}$ is the four-dimensional simply connected nilpotent Lie group called the Engel group~\cite{mont}. The quotient mapping in coordinates is
$$
\mapto{\pi}{(x, y, z, v, w)}{(x, y, z, v)}, \qquad G \to H.
$$
The Engel algebra is 
\begin{align*}
&\mathfrak h = \pi_*\g = \spann(Y_1, Y_2, Y_3, Y_4), \qquad Y_i = \pi_* X_i, \quad i = 1, \dots, 4, \\
&[Y_1,Y_2] = Y_3, \qquad [Y_1, Y_3] = Y_4, \qquad \ad Y_4 = 0.
\end{align*}
The left-invariant SR problem on the Engel group
\begin{align}
&\dot h = u_1 Y_1 + u_2 Y_2, \qquad h \in H, \quad u = (u_1, u_2) \in \R^2, \label{en1}\\
&h(0) = \Id, \qquad h(t_1) = h_1, \label{en2}\\
&\int_0^{t_1} \sqrt{u_1^2 + u_2^2} \, dt \to \min \label{en3}
\end{align}
was studied in detail in~\cite{engel, engel_conj, engel_cut, engel_synth}. In particular, it was shown that if an extremal control $u(t)$ in problem~\eq{en1}--\eq{en3} is non-periodic (i.e., if $\lam \in C_3$), then it is optimal for $t \in [0, + \infty)$.

Further, if a projection $h_t = \pi(g_t) \in H$ is optimal for problem~\eq{en1}--\eq{en3} on the Engel group $H$, then the trajectory $g_t \in G$ is optimal for problem~\eq{pr1}--\eq{pr3} on the Cartan group.

Thus geodesics $\Exp(\lam, t)$, $\lam \in C_3$, $t \in [0, + \infty)$, are optimal for problem~\eq{pr1}--\eq{pr3}. So we have $\tcut(\lam) = \tconj(\lam) = + \infty = \tmax(\lam)$ for $\lam \in C_3$. 
\end{proof}

\begin{theorem}\label{th:tconjC6}
If $\lam \in C_6$, then $\tconj(\lam) = \tmax(\lam)$.
\end{theorem}
\begin{proof}
Let $\blam = (\bth, \bh_3, \bh_4, \bh_5) \in C_6$, where $\bh_3 \neq 0$, $\bh_4 = \bh_5 = 0$. Take any $\hmu > 0$ and consider a continuous curve
\begin{align*}
&\lam^\mu = (\bth, \bh_3, h_4, h_5), \qquad h_4 = \mu, \quad h_5 = 0, \qquad \mu \in [0, \hmu],\\
&\lam^0 = \blam, \qquad \lam^\mu \in C_2, \quad \mu \in (0, \hmu].
\end{align*}

Notice that the function $\mu \mapsto \tmax(\lam^\mu)$, $\mu \in [0, \hmu]$, is continuous since
$$
\lim_{\mu \to + 0} \tmax(\lam^\mu) = \lim_{\mu \to + 0} \frac{2k}{\sqrt \a} p_1^V(k) = \frac{4}{|\bh_3|} p_1^V(0) = \tmax(\lam^0).
$$

Take any sufficiently small $\eps> 0$ and define a continuous family of geodesics:
\begin{align*}
&g_t^\mu = \Exp(\lam^\mu, t), \qquad t \in (0, t_1^\mu], \\
&t_1^\mu = \tmax(\lam^\mu) - \eps.
\end{align*}
If $\mu \in (0, \hmu]$, then $\lam^\mu \in C_2$, and the geodesic $g_t^\mu$ does not contain conjugate points by Th.~\ref{th:tconjC2}. 

Conjugate instants along $g_t^0$ are isolated one from another (see Sec.~\ref{sec:conj_hom}), thus we can choose an arbitrarily small $\eps> 0$ such that the instant $t_1^0 = \tmax(\blam) - \eps$ is not conjugate along $g_t^0$. Then, by Th.~\ref{th:conj_hom}, the geodesic $g_t^0$,  $t \in (0, t_1^0]$, also does not contain conjugate points.

Since $\eps>0$ can be chosen arbitrarily small, then the geodesic $g_t^0$,  $t \in (0, \tmax(\blam))$, does not contain conjugate points.

It was proved in \cite{max3} that the instant $\tmax(\blam)$ is conjugate for $g_t^0$, thus $\tconj(\blam) = \tmax(\blam)$.
\end{proof}

\section{Cases when $\tconj(\lam) = \tmax(\lam)$}\label{sec:tconj=} 
In this section we present cases when the lower bound of the first conjugate time given by Th.~\ref{th:tconj} turns into equality.

\subsection{Cases when $\tconj(\lam) = \tmax(\lam) = + \infty$}
\begin{proposition}\label{propos:tconj=C3}
If $\lam \in C_3 \cup C_4 \cup C_5 \cup C_7$, then $\tconj(\lam) = \tmax(\lam) = \tcut(\lam)=+ \infty$.
\end{proposition}
\begin{proof}
Apply Theorems~\ref{th:tconjC47}, \ref{th:tconjC3}.
\end{proof}
 
\subsection{Cases when $\tconj(\lam) = \tmax(\lam) < + \infty$}
\begin{proposition}\label{propos:tconj=C1}
Let $\lam = (\f, k, \a, \b) \in C_1$. The equality $\tconj(\lam) = \tmax(\lam) < + \infty$ holds in the following cases:
\begin{itemize}
\item[$(1)$]
$k \in \{k_0, k_1\}$,
\item[$(2)$]
$k \in (0, k_1) \cup (k_0, 1)$, $\cn \tau = 0$, $\tau = \sqrt \a (\f + \tmax(\lam)/2)$,
\item[$(3)$]
$k \in (k_1, k_0)$, $\sn \tau = 0$, $\tau = \sqrt \a (\f + \tmax(\lam)/2)$.
\end{itemize}
\end{proposition}

Recall that $k_1 \approx 0.8$ and $k_0 \approx 0.9$ are the only roots of the equation $p_z^1(k) = p_V^1(k)$, $k \in (0, 1)$, see~\cite{max3}.

\begin{proof}
We prove that the instant $t_1 = \tmax(\lam)$ is conjugate in cases $(1)$--$(3)$.

By virtue of equalities~\eq{J2=a0x}--\eq{a1}, we have 
$$
J = (1-\xi) f_V a_{01} - \frac{\xi(1-k^2 \xi)}{k^2} f_z a_{21}, \qquad p = \sqrt \a t_1 /2.
$$

\medskip

If $k \in \{k_0, k_1\}$, then $f_V(p,k) = f_z(p,k) = 0$, thus $J = 0$. So the instant $t_1$ is conjugate.

\medskip

If $\cn \tau = 0$, then $\xi = 1$, thus
$$
J = - \frac{1-k^2}{k^2} f_z a_{21}.
$$
For $k \in (0, k_1) \cup (k_0, 1)$ we have from~\cite{max3}
$$
\tmax(\lam) = \frac{2}{\sqrt \a} p_1^z(k),
$$
thus $f_z(p) = f_z(p_1^z(k)) = 0$. So $J = 0$.

\medskip

If $\sn \tau = 0$, then $\xi = 0$, thus
$$
J = f_V a_{01}.
$$
For $k \in (k_1, k_0)$ we have from~\cite{max3}
$$
\tmax(\lam) = \frac{2}{\sqrt \a} p_1^V(k),
$$
thus $f_V(p) = f_z(p_1^V(k)) = 0$. So $J = 0$.

\medskip
We proved that in all cases (1)--(3) the instant $t_1 = \tmax(\lam)$ is conjugate. By Th.~\ref{th:tconj}, all instants $t \in (0, \tmax(\lam))$ are not conjugate, thus $\tconj(\lam) = \tmax(\lam)$.
\end{proof}

\begin{remark}
The condition $\cn \tau = 0 $ (resp. $\sn \tau = 0 $) in Propos.~$\ref{propos:tconj=C1}$ means that the corresponding inflectional elastic arc $(x_t, y_t)$ is centered at inflection point (resp. at vertex), see~\cite{max2}.
\end{remark}

\begin{proposition}\label{propos:tconj=C2}
Let $\lam = (\psi, k, \a, \b) \in C_2$. The equality $\tconj(\lam) = \tmax(\lam) < + \infty$ holds if
$$
\sn^2 \tau \in\{ 0, 1\}, \qquad \tau = \sqrt \a (\psi + \tmax(\lam)/(2k)).
$$
\end{proposition}
\begin{proof}
See Th.~\ref{th:tconjC2x}.
\end{proof}

\begin{remark}
The condition $\sn^2 \tau \in\{ 0, 1\}$ in Propos.~$\ref{propos:tconj=C2}$ means that the corresponding non-inflectional elastic arc $(x_t, y_t)$ is centered at vertex, see~\cite{max2}.
\end{remark}

\begin{proposition}\label{propos:tconj=C6}
If $\lam  \in C_6$, then $\tconj(\lam) = \tmax(\lam) < + \infty$.
\end{proposition}
\begin{proof}
See Th.~\ref{th:tconjC6}.
\end{proof}

\begin{remark}
There is a numerical evidence that the equality $\tconj(\lam) = \tmax(\lam)$ holds only in the cases listed in Propositions~$\ref{propos:tconj=C3}$--$\ref{propos:tconj=C6}$.
\end{remark}

We plot elasticae terminating at the instant $t_1 = \tconj(\lam) = \tmax(\lam)$ in the following cases:
\begin{itemize}
\item
$\lam \in C_1$, $k = k_0$ (Fig.~\ref{fig:tconj=C18}), we draw the initial elastica plus elasticae obtained from it by the reflections $\eps^1$, $\eps^2$, $\eps^3$,
\item
$\lam \in C_1$, $k = k_1$ (Fig.~\ref{fig:tconj=C1k1}),
\item
$\lam \in C_1$, $\cn \tau = 0$, $k \in(0, k_1)$ and $k \in  (k_0, 1)$ (Figs.~\ref{fig:tconj=C1k<k1} and~\ref{fig:tconj=C1k>k0}),  
\item
$\lam \in C_1$, $\sn \tau = 0$, $k \in(k_1, k_0)$ (Fig.~\ref{fig:tconj=C1k0}),  
\item
$\lam \in C_2$, $\sn \tau = 0$ (Fig.~\ref{fig:tconj=C2sn}),  
\item
$\lam \in C_2$, $\cn \tau = 0$ (Fig.~\ref{fig:tconj=C2cn}),  
\item
$\lam \in C_6$ (Fig.~\ref{fig:tconj=C6}).  
\end{itemize}

\twofiglabelsize{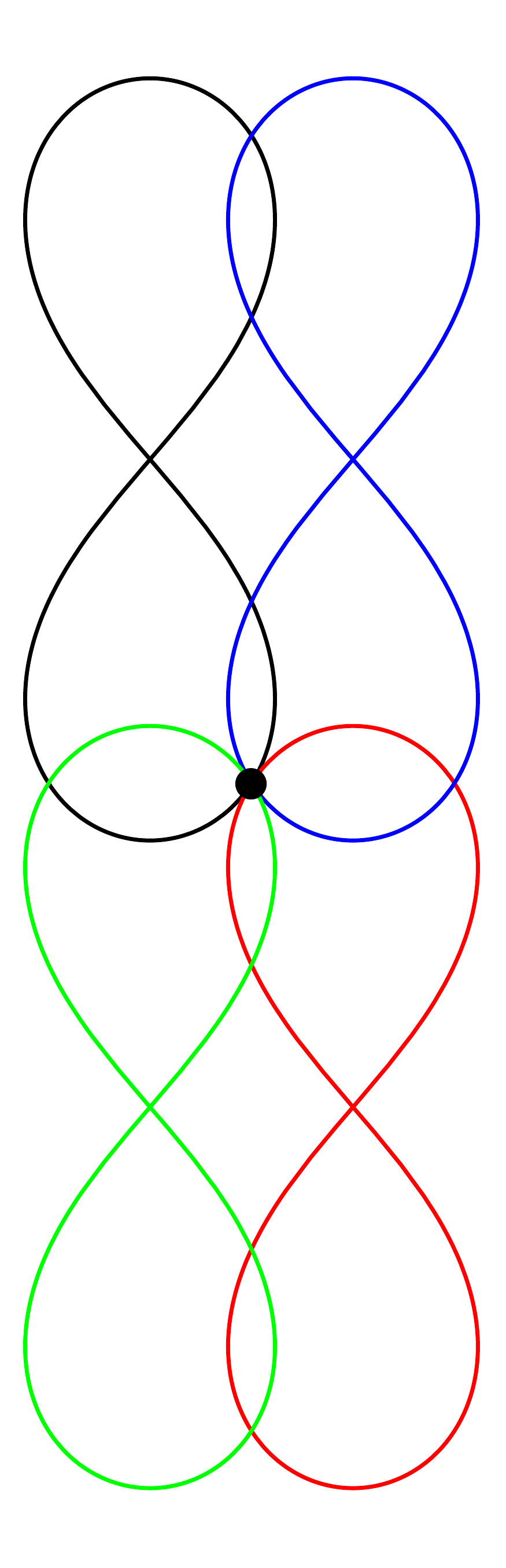}{$\lam \in C_1$, $k = k_0$}{fig:tconj=C18}{7}
{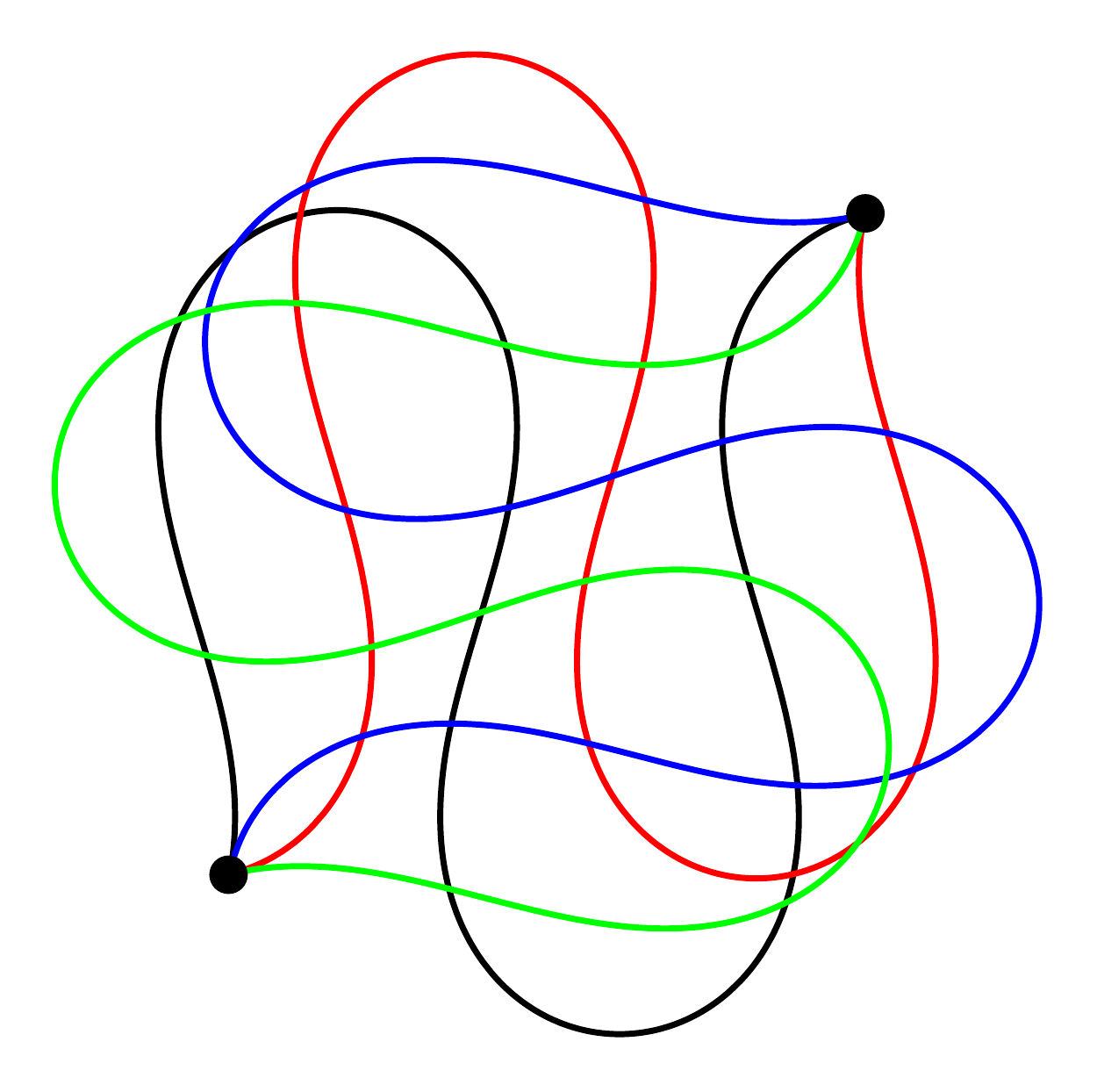}{$\lam \in C_1$, $k = k_1$}{fig:tconj=C1k1}{5}

\twofiglabelsize{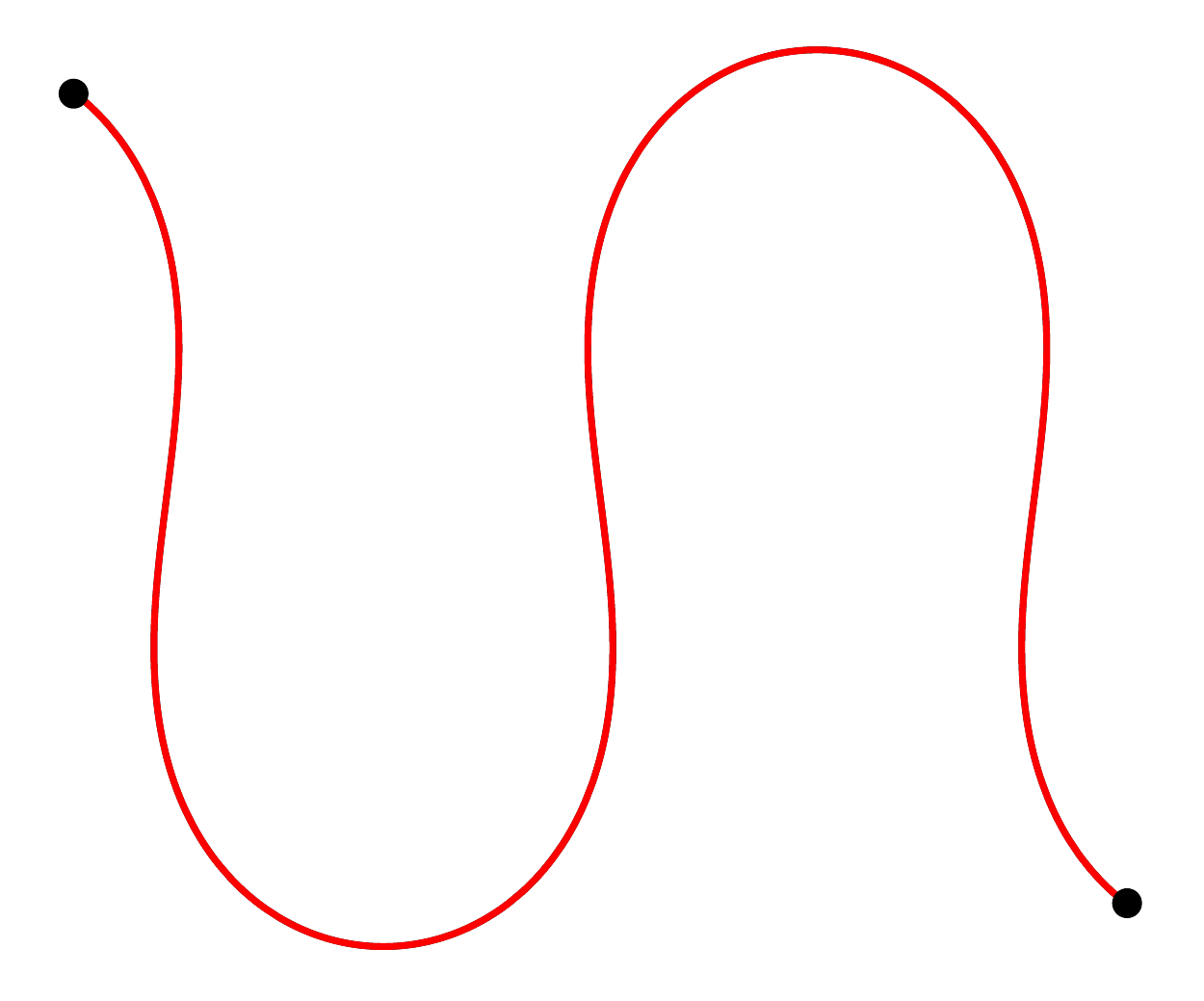}{$\lam \in C_1$, $\cn \tau = 0$, $k \in(0, k_1)$}{fig:tconj=C1k<k1}{4} 
{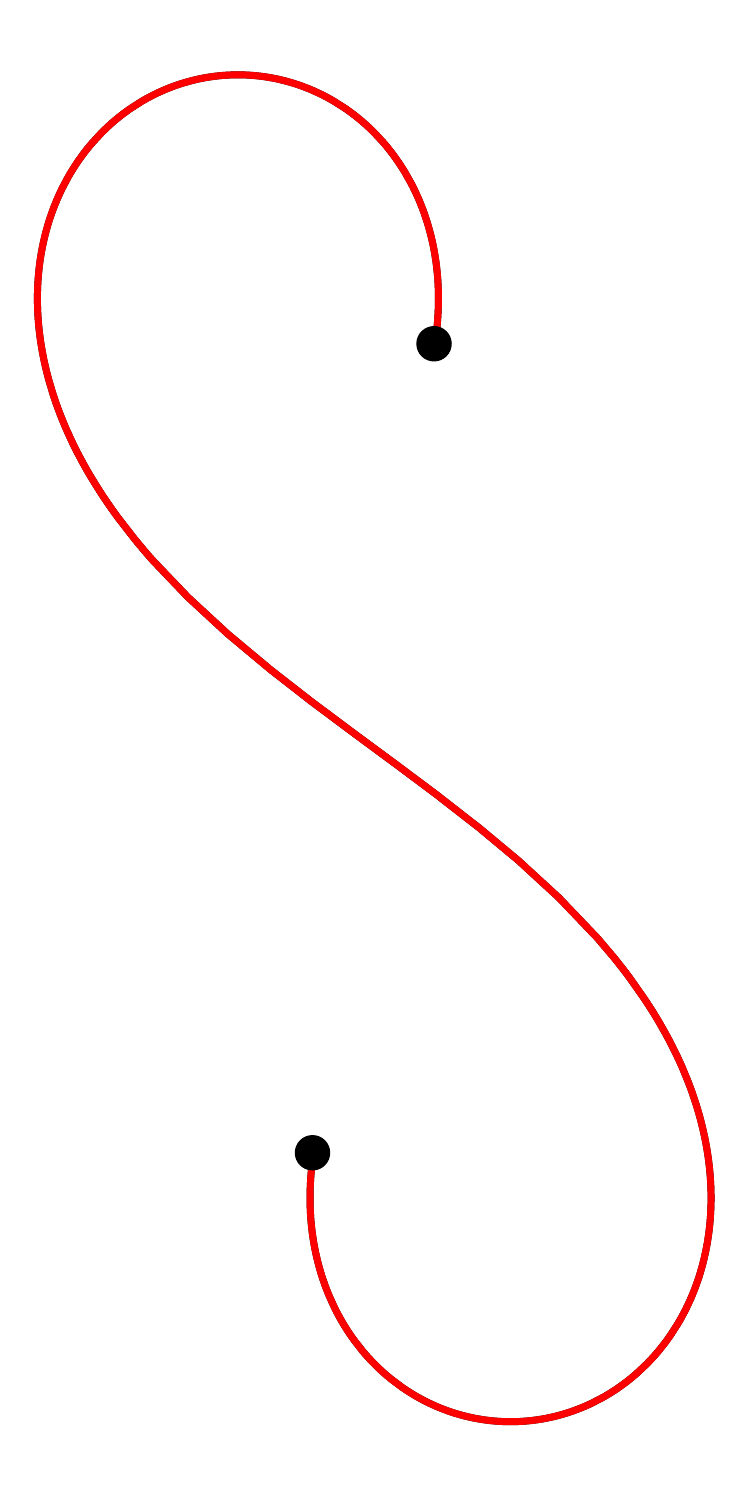}{$\lam \in C_1$, $\cn \tau = 0$, $k \in  (k_0, 1)$}{fig:tconj=C1k>k0}{6}

\twofiglabelsize{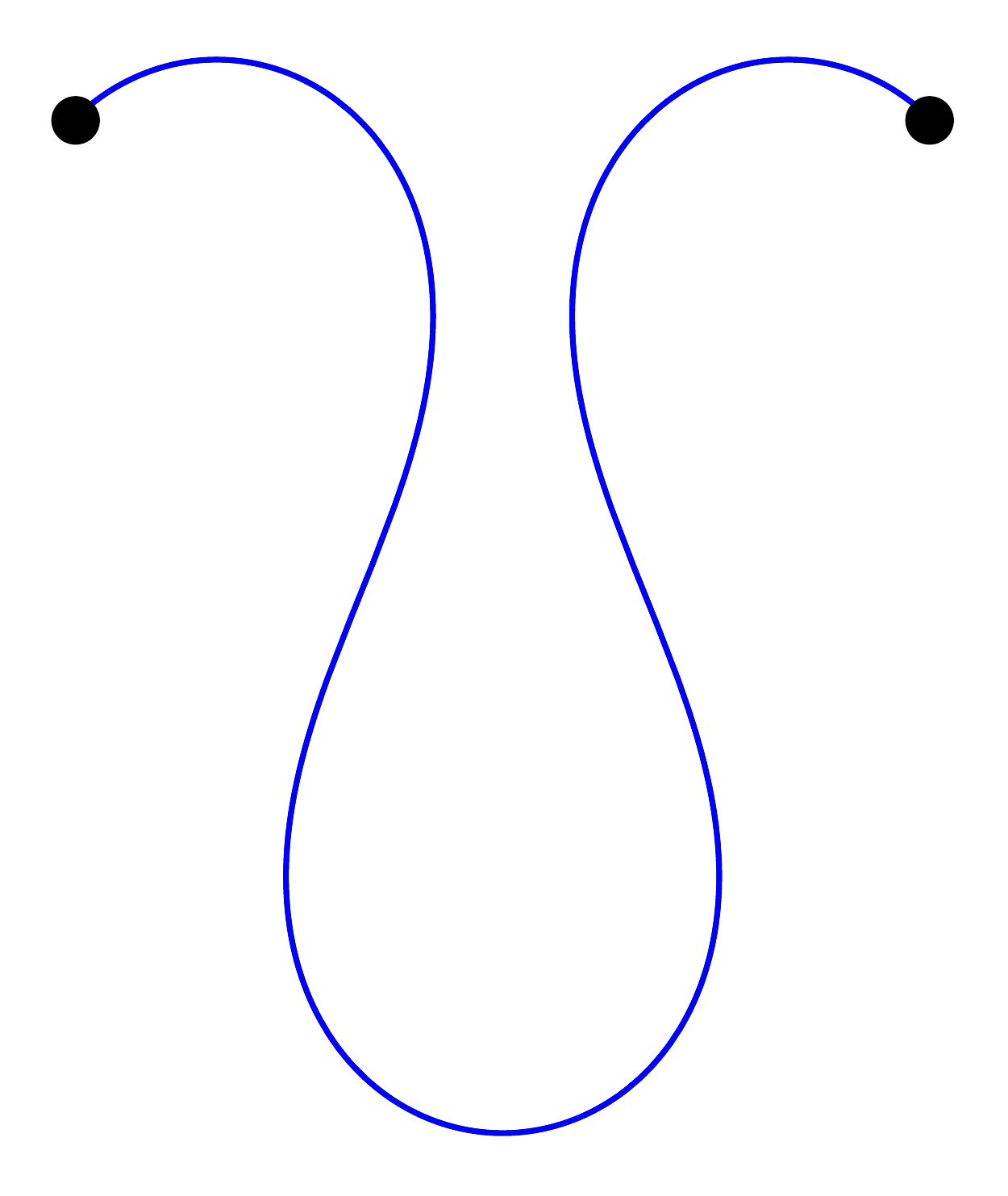}{$\lam \in C_1$, $\sn \tau = 0$}{fig:tconj=C1k0}{5}
{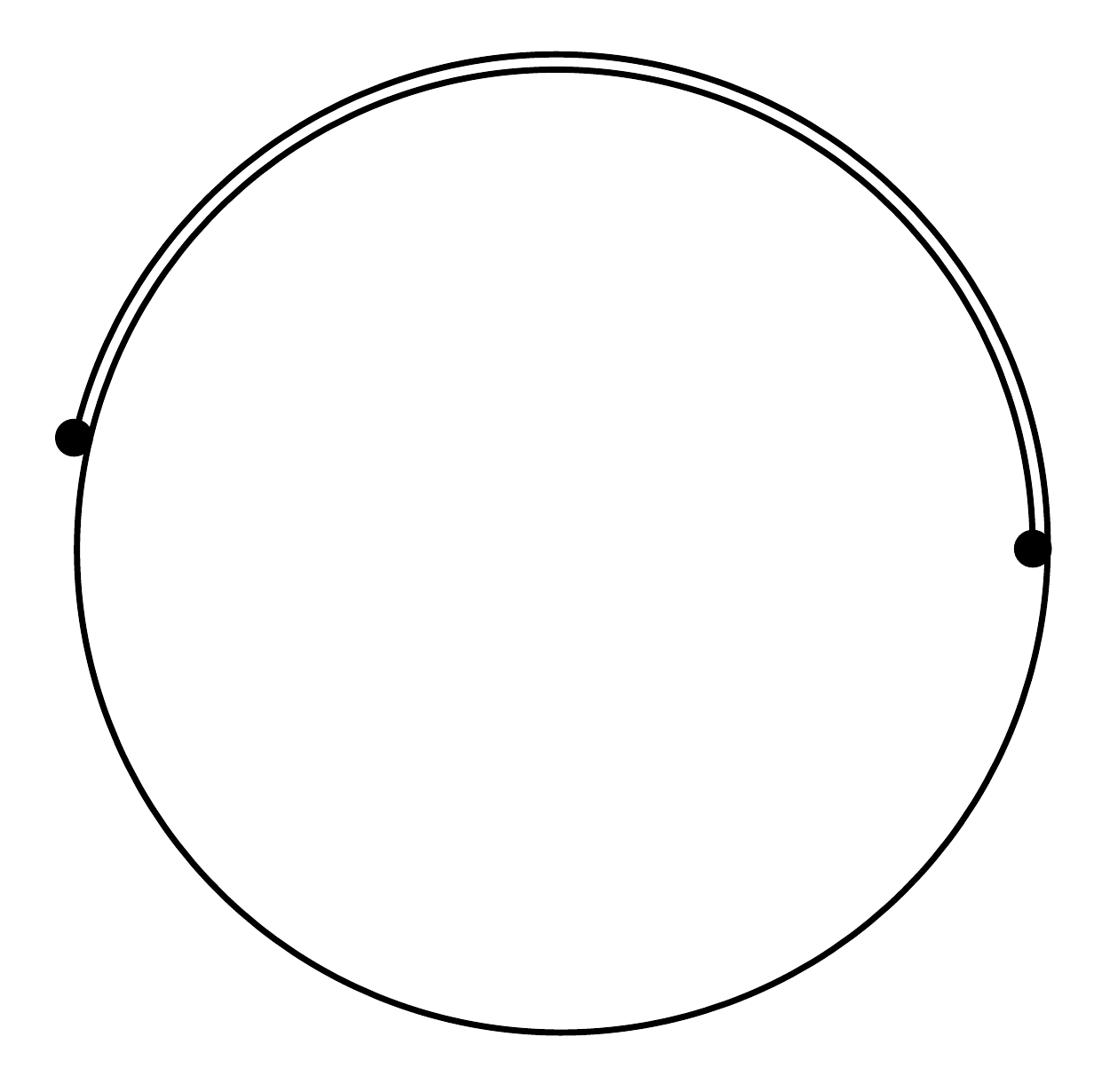}{$\lam \in C_6$}{fig:tconj=C6}{5}

\twofiglabelsize{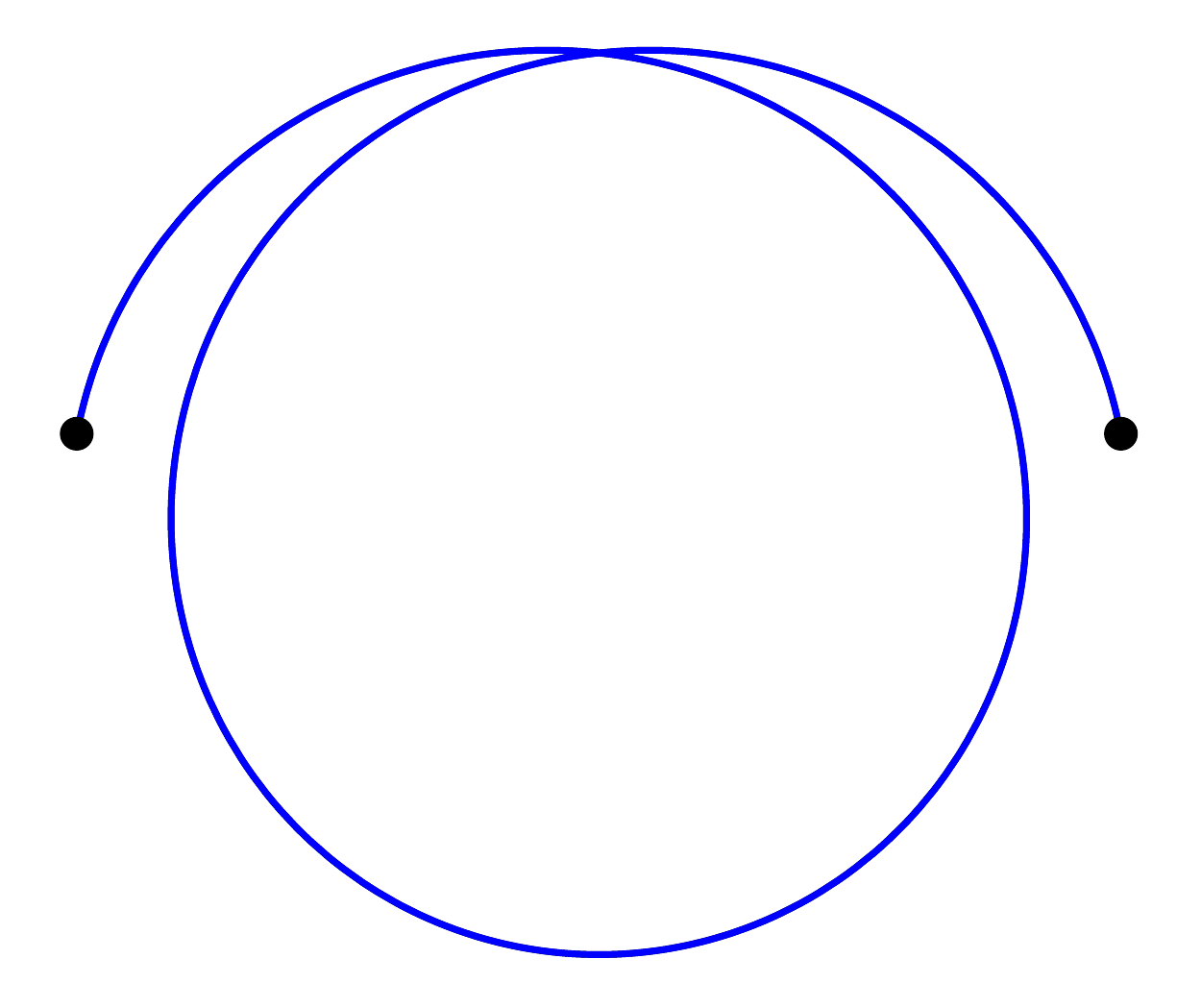}{$\lam \in C_2$, $\sn \tau = 0$ }{fig:tconj=C2sn}{5}
{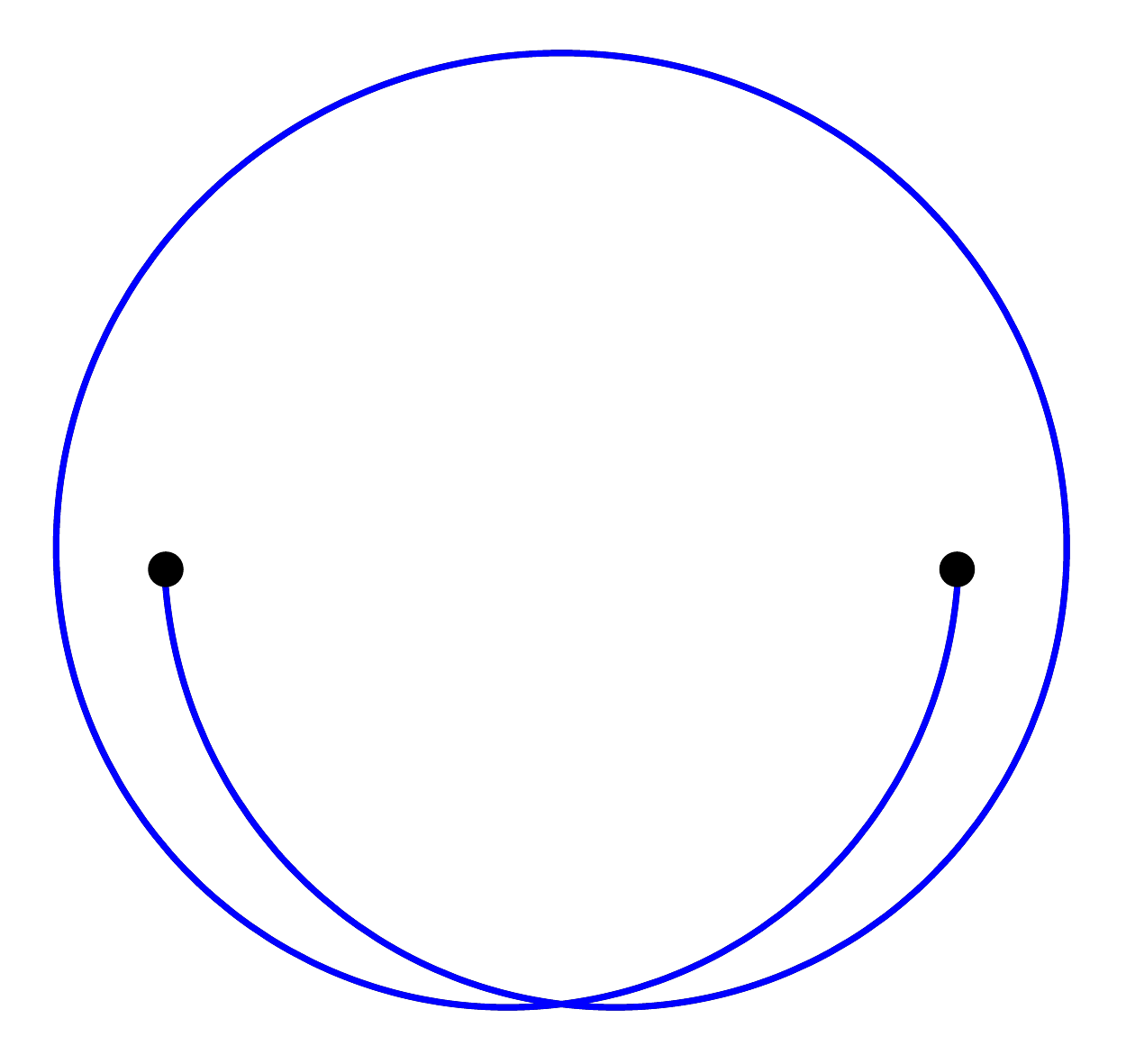}{$\lam \in C_2$, $\cn \tau = 0$ }{fig:tconj=C2cn}{5}

\begin{remark}
Notice that despite the equality $\tconj(\lam) = \tmax(\lam)$ in Proposition~$\ref{propos:tconj=C1}$, items $(2)$, $(3)$  and in Propositions~$\ref{propos:tconj=C2}$,  $\ref{propos:tconj=C6}$, the corresponding point $\Exp(\lam, \tmax(\lam))$ is not a Maxwell point but a limit of pairs of symmetric Maxwell points, see~\cite{max3}. On the contrary, in Proposition~$\ref{propos:tconj=C1}$, item $(1)$, the  point $\Exp(\lam, \tmax(\lam))$ is a Maxwell point, thus in Figs.~$\ref{fig:tconj=C18}$, 
$\ref{fig:tconj=C1k1}$ we draw four symmetric elasticae corresponding to the group of reflections $\{\Id, \eps^1, \eps^2, \eps^3\}$.
\end{remark}

\section{Continuity of $\tconj(\lam)$ at infinite values}\label{sec:tconjcont}

\begin{proposition}\label{th:tconjcont}
Let $\lam_n, \, \blam \in C$ and $\lam_n \to \blam$ as $n \to \infty$. We have $\tconj(\lam_n) \to \tconj(\blam) \to + \infty$ in the following cases:
\begin{itemize}
\item[$(1)$]
$\lam_n \in C_1$, $\blam \in C_3 \cup C_5$, 
\item[$(2)$]
$\lam_n \in C_2$, $\blam \in C_3 \cup C_5$, 
\item[$(3)$]
$\blam \in C_7$. 
\end{itemize}
\end{proposition}
\begin{proof}
(1) We have $\blam = (\bar \theta, \bar c, \bar \a, \bar \b) \in C_3 \cup C_5$, thus $\bar E = \frac{\bar c^2}{2} - \bar \a \cos(\bar \theta - \bar \beta) = \bar \a > 0$. Since $\lam_n = (\theta_n, c_n, \a_n, \b_n) \to \blam$, then $E_n + \a_n \to \bar E + \bar \a = 2 \bar \a > 0$, thus $k_n = \sqrt{\frac{E_n + \a_n}{2 \a_n}} \to 1$. So $\tmax(\lam_n) \to + \infty$, thus $\tconj(\lam_n) \to + \infty = \tconj(\blam)$.

\medskip

(2) Similarly to item (1), we have $E_n + \a_n \to 2 \bar \a > 0$, thus $k_n = \sqrt{\frac{2 \a_n}{E_n + \a_n}} \to 1$. Then $\tmax(\lam_n) \to + \infty$, thus $\tconj(\lam_n) \to + \infty = \tconj(\blam)$.

\medskip

(3) We have $\blam = (\bar \theta, \bar c, \bar \a, \bar \b) \in C_7$ with $\bar c = \bar \a = 0$. If $\lam_n = (\theta_n, c_n, \a_n, \b_n) \to \blam$, then $\a_n \to \bar \a = 0$, thus $\tmax(\lam_n) = \frac{2}{\sqrt{\a_n}} p_1(k_n) \to + \infty$ in the case $\lam_n \in C_1$. The cases $\lam _n \in \cup_{i=2}^7 C_i$ are considered similarly. 
\end{proof}

\begin{remark}
Let $\lam = (\f, k, \a, \b) \in C_1$ with $\cn \tau = 0$, $\tau = \sqrt{\a}(\f + \tmax(\lam)/2)$, $k \in (0, k_1)$. Then $\tconj(\lam) = \tmax(\lam)$ by Propos.~$\ref{propos:tconj=C1}$. Let $\a = \const$, $k \to + 0$, then $\lam \to \blam \in C_4$. We have 
$$
\lim_{k \to + 0} \tmax(\lam) = \frac{2}{\sqrt \a} p_1^z(0), \qquad p_1^z(0) \in(\pi, 3\pi/2)
$$
by Propos.~$2.1$~\cite{max3}. Thus
$$
\lim_{k \to + 0} \tconj(\lam)  < + \infty = \tconj(\blam).
$$
So $\tconj(\lam)$ is discontinuous at points $\lam \in C_4$. This is related to the fact that the geodesic $\Exp(\lam, t)$, $\lam \in C_4$, is abnormal.
\end{remark}

\section{Two-sided bounds of $\tconj(\lam)$}\label{tconj2side} 

There is a numerical evidence of the following two-sided bounds of the first conjugate time:
\begin{align}
&\lam \in C_1 \then \tmax(\lam)  \leq \tconj(\lam) \leq  \frac{2}{\sqrt \a} \max(p_1^z(k), p_1^V(k)), \label{tconj2sideC1}\\ 
&\lam \in C_2 \then \tmax(\lam) \leq  \tconj(\lam) \leq  \frac{4}{\sqrt \a} k K. \label{tconj2sideC2}
\end{align}
We believe that the upper bounds can be proved by methods of this paper.

See the plots justifying bounds~\eq{tconj2sideC1} and~\eq{tconj2sideC2}  of the function $k \mapsto \tconj(\lam)$ at Figs.~\ref{fig:tconj2sideC1} and \ref{fig:tconj2sideC2} respectively (for $\a = 1$).

\twofiglabelsize{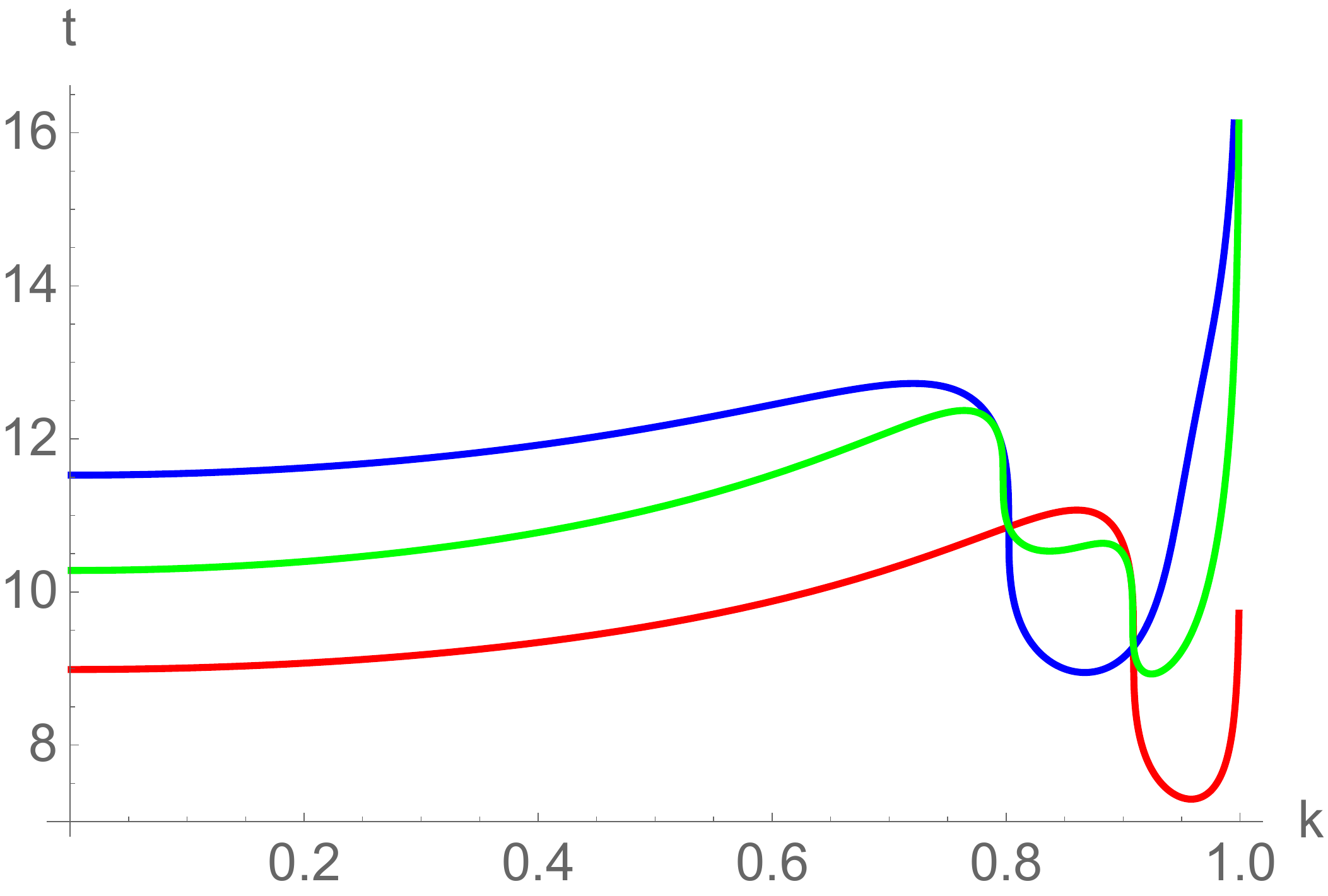}{Bound \eq{tconj2sideC1}: $k \mapsto \tconj(\lam)$}{fig:tconj2sideC1}{3.5}
{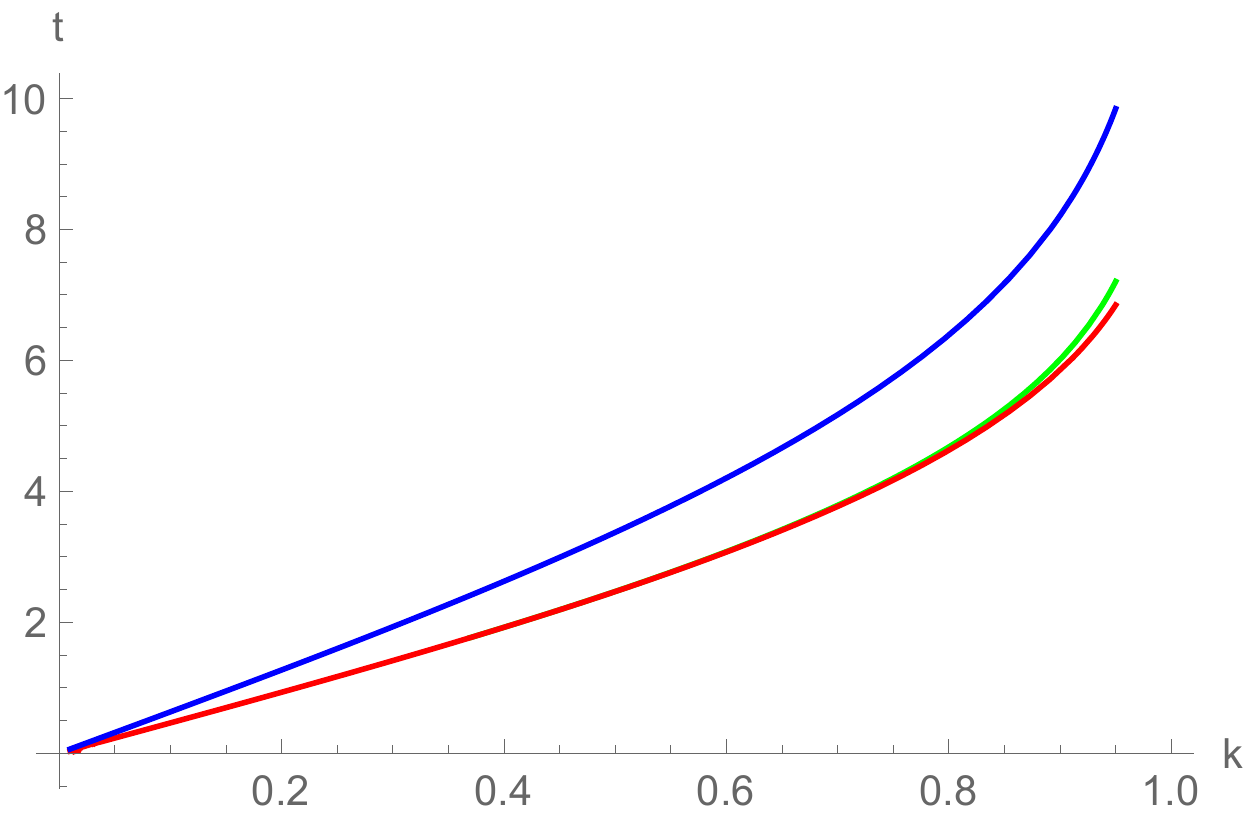}{Bound \eq{tconj2sideC2}: $k \mapsto \tconj(\lam)$}{fig:tconj2sideC2}{3.5}

Figure \ref{fig:tconj2sidephiC1} presents a plot of the function $\f \mapsto \tconj(\lam)$ for $\lam \in C_1$. It shows that the first conjugate time $\tconj(\lam)$ is not preserved by the flow of  pendulum~\eq{pend} (i.e., the vertical part of the Hamiltonian vector field $\vH$), unlike the first Maxwell time $\tmax(\lam)$.  Figure \ref{fig:tconj2sidephiC1} confirms periodicity of the function $\f \mapsto \tconj(\lam)$, which is a manifestation of invariance of the first conjugate time w.r.t. the reflection $\eps^3$, see Cor.~\ref{cor:cutconjsym}.

\onefiglabelsize{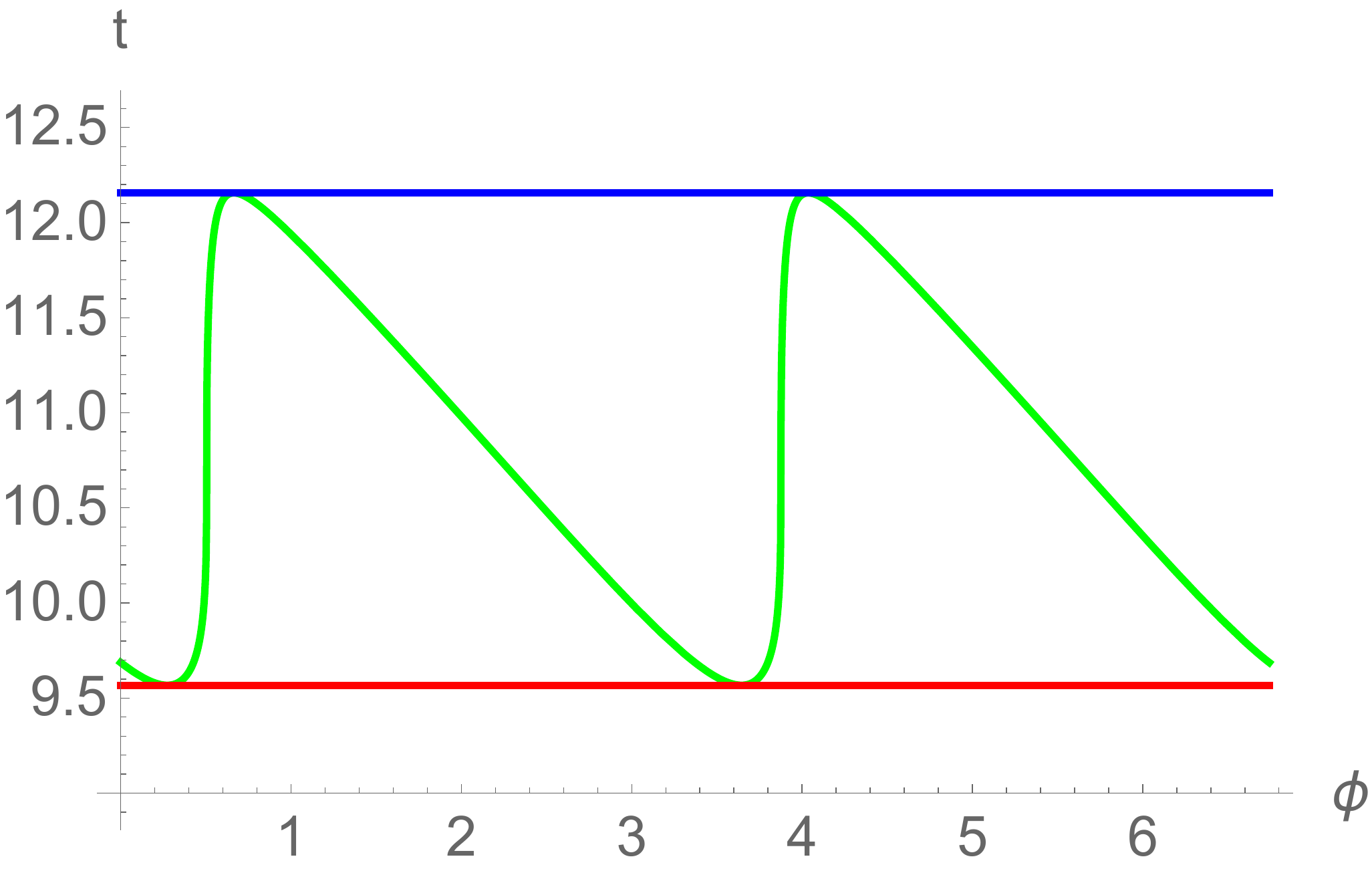}{Bound \eq{tconj2sideC1}: $\f \mapsto \tconj(\lam)$}{fig:tconj2sidephiC1}{0.5}

There is a numerical evidence that bound \eq{tconj2sideC1} is exact. Moreover, the both bounds \eq{tconj2sideC1} and \eq{tconj2sideC2} contain just one value of $\tconj(\lam)$, so they can be used for reliable numerical evaluation of the first conjugate time.

\section{Conclusion}

In this work we used several methods for bounding conjugate points:
\begin{itemize}
\item
direct estimate of Jacobian of the exponential mapping via comparison functions (for $\lam \in C_1$),
\item
homotopy invariance of the index of the second variation (for $\lam \in C_2 \cup C_6$),
\item
projection to lower-dimensional SR minimizers (for $\lam \in C_3 \cup C_4 \cup C_5 \cup C_7$).
\end{itemize}
We believe that these methods can be useful for bounding conjugate points in other SR problems.

\medskip

Using the estimate of cut time obtained in work~\cite{max3} (Theorem~\ref{th:tcut_bound}) and the estimate of conjugate time proved in this work (Theorem~\ref{th:tconj}), one can prove Conjecture 1 and get a description of global structure of the exponential map in the SR problem on the Cartan group. So we can reduce this problem to solving a system of algebraic equations. This will be the subject of  forthcoming works.

\appendix
\section[Appendix: Explicit formulas for coefficients of Jacobian]
{Appendix: \\Explicit formulas for coefficients of Jacobian}
In this appendix we present explicit formulas for the coefficients $a_{01}$, $a_{12}$ of the Jacobian $J_1$~\eq{J1C2} in the domain $C_2$.

\begin{align}
&a_{01} = \sum_{j=0}^3 \sum_{i=0}^j c_{i, j-i} F^i(u_1) E^{j-i}(u_1), \label{a01C2} \\
&c_{00} = - 2 k^2 \sin^2 u_1 \cos^2 u_1 \sqrt{1-k^2 \sin^2 u_1}, \nonumber \\
&c_{01} = \cos u_1 \sin u_1 (4+k^2(1-6\sin^2 u_1)), \nonumber \\
&c_{02} = - 3 \sqrt{1-k^2 \sin^2 u_1} (1-2\sin^2 u_1), \nonumber\\
&c_{11} = (2-k^2) \sqrt{1-k^2 \sin^2 u_1}, \nonumber\\
&c_{20} = (1-k^2) \sqrt{1-k^2 \sin^2 u_1} (1-2\sin^2 u_1), \nonumber\\
&c_{03} = - 2 \cos u_1 \sin u_1, \nonumber\\
&c_{12} = (2-k^2) \cos u_1 \sin u_1, \nonumber\\
&c_{21} = 2(1-k^2) \cos u_1 \sin u_1, \nonumber\\
&c_{30} = -(1-k^2)(2-k^2) \cos u_1 \sin u_1, \nonumber
\end{align}

\begin{align}
&a_{21} = \sum_{j=0}^5 \sum_{i=0}^j d_{i,j-i} F^i(u_1) E^{j-i}(u_1), \label{a21C2}\\
&d_{00} = - 6 k^7 \sin^3 u_1 \cos ^3 u_1, \nonumber \\
&d_{01} = 20 k^5 \sin^2 u_1 \cos^2 u_1  \sqrt{1-k^2 \sin^2 u_1}, \nonumber \\
&d_{10} = - 6 k^5 (2-k^2) \sin^2 u_1 \cos^2 u_1 \sqrt{1-k^2 \sin^2 u_1}, \nonumber \\
&d_{02} = - 2 k^3 \sin u_1 \cos u_1(12-k^2(1 + 10 \sin^2 u_1)), \nonumber \\
&d_{11} = \frac{k^3}{2} \sin u_1 \cos u_1(32 - 8 k^2(1+6\sin^2 u_1) + 3 k^4(1+8\sin^2 u_1)), \nonumber \\
&d_{20} = \frac{k^3}{2} \sin u_1 \cos u_1(16 + 3 k^6 \sin^2 u_1 + k^4 (9-8 \sin^2 u_1)-4k^2(7-2 \sin^2 u_1)), \nonumber \\
&d_{03} = 8 k (2-k^2) \sqrt{1-k^2 \sin^2 u_1}, \nonumber \\
&d_{12} = - \frac{k}{2}(32 - 32 k^2 + 15 k^4) \sqrt{1-k^2 \sin^2 u_1}, \nonumber \\
&d_{21} = - \frac{k}{2}(32 - 48 k^2 + 10k^4+3k^6) \sqrt{1-k^2 \sin^2 u_1}, \nonumber \\
&d_{30} = \frac k2 (32 - 64 k^2 + 41 k^4 - 9 k^6) \sqrt{1-k^2 \sin^2 u_1}, \nonumber \\
&d_{04} = -10k^3 \sin u_1 \cos u_1, \nonumber  \\
&d_{13} = 12 k^3(2-k^2) \sin u_1 \cos u_1,\nonumber  \\
&d_{22} = - \frac 32 k^3(8-8k^2+3k^4) \sin u_1 \cos u_1, \nonumber \\
&d_{31} = - \frac{k^3}{2}(16-24k^2+6k^4+k^6) \sin u_1 \cos u_1, \nonumber \\
&d_{40} = \frac 32 k^3(1-k^2)(2-k^2)^2 \sin u_1 \cos u_1, \nonumber \\
&d_{05} = 4 k \sqrt{1-k^2 \sin^2 u_1}, \nonumber \\
&d_{14} = - 6 k (2-k^2) \sqrt{1-k^2 \sin^2 u_1}, \nonumber \\
&d_{23} = k (8-8k^2 +3k^4) \sqrt{1-k^2 \sin^2 u_1}, \nonumber \\
&d_{32} = \frac k2  (16 -24k^2+6k^4+k^6)\sqrt{1-k^2 \sin^2 u_1}, \nonumber \\
&d_{41} = -3k(1-k^2)(2-k^2)^2 \sqrt{1-k^2 \sin^2 u_1}, \nonumber \\
&d_{50} = \frac k2 (1-k^2)(2-k^2)^3 \sqrt{1-k^2 \sin^2 u_1}.\nonumber 
\end{align}
Here $F(u_1)$ and $E(u_1)$ are elliptic integrals of the first and second kinds~\cite{whit_vatson}.

\end{document}